\documentclass[10pt,leqno,a4paper]{article}
\usepackage{amssymb,bbm,amsmath,
amsthm,
amscd}
\usepackage{vmargin}
\usepackage[dvipsnames]{xcolor}
\usepackage{
 array}
\catcode`\@=11
\newbox\bz@
\newdimen\bdimz@
\def\linethrough#1{\setbox\bz@=\hbox{#1}%
\bdimz@=\ht\bz@ \divide\bdimz@ by 5 \advance\bdimz@ by -\dp\bz@ \ht\bz@=\bdimz@
\leavevmode\hbox{$\overline{\overline{\box\bz@}}$\relax}}
\catcode`\@=12

\def \new #1 {\textcolor{Maroon}{\underline{#1}} }

\usepackage[colorlinks=true]{hyperref}
\newcommand{\BEQ} {\begin{equation} }
  \newcommand{\EEQ} {\end{equation} }

\newcommand{\BTHM}{\begin{theorem}}
  \newcommand{\ETHM}{\end{theorem}}
\hypersetup{urlcolor=blue, citecolor=red}

\newtheorem{lemma}{Lemma}[section]

\newtheorem{theorem}{Theorem}[section]
\newtheorem{definition}{Definition}[section]
\newtheorem{proposition}{Proposition}[section]
\newtheorem{remark}{Remark}[section]

\newcommand{\tore}{\mathbb{T}_3}

\newcommand{\expk}{e^{-i\mathbf{k \cdot x}} }

\newcommand{ \vit}{\hbox{\bf u}}

\newcommand{ \vittest }{\hbox{\bf v}}
\newcommand{ \wit}{\hbox{\bf w}}
\newcommand{ \bc}{\mathbf{c}}
\newcommand{ \bd}{\mathbf{d}}

\newcommand{ \bu}{\hbox{\bf u}}
\newcommand{ \bB}{\hbox{\bf B}}
\newcommand{ \bb}{\hbox{\bf b}}
\newcommand{ \bx}{\hbox{\bf x}}
\newcommand{ \bh}{\hbox{\bf h}}

\newcommand{ \bn}{D_N}
\newcommand{ \obu}{\overline{\bu}}
\newcommand{ \bw}{\hbox{\bf w}}

\newcommand{ \obw}{\overline{\bw}}
\newcommand{ \bH}{\mathbf{H}}

\newcommand{ \bff}{\hbox{\bf f}}
\newcommand{ \Z}{\mathbbm{Z}}
\newcommand{\moy} {\overline {\vit} }
\newcommand{\moyb} {\overline {\bB} }
\newcommand{\g} {\nabla }
\newcommand{\p} {\partial}
\newcommand{\x} {{\bf x}}
\newcommand{\N}{\mathbb{N}}
\newcommand{\R}{\mathbb{R}}
\newcommand{\ind}{m}
\newcommand{\witt}{\overline \wit}

\makeatletter

\@addtoreset{equation}{section}

\begin{document}

\title{\bf Convergence of approximate deconvolution models to the
  mean Magnetohydrodynamics Equations: Analysis of two models}

\author{Luigi C. Berselli \thanks{Dipartimento di Matematica Applicata
    Universit\`a di Pisa, Via F. Buonarroti 1/c, I-56127, ITALY,
    berselli@dma.unipi.it,
    http://users.dma.unipi.it/berselli}
  \and
  Davide Catania
  \thanks{Dipartimento di Matematica, Universit\`a di Brescia,
    Via Valotti 9, 25133 Brescia, ITALY, \mbox{davide.catania@ing.unibs.it},
    http://www.ing.unibs/$\sim$davide.catania/}
  \and
  Roger Lewandowski
  \thanks{IRMAR, UMR 6625, Universit\'e Rennes 1, Campus Beaulieu,
    35042 Rennes cedex FRANCE, Roger.Lewandowski@univ-rennes1.fr,
    http://perso.univ-rennes1.fr/roger.lewandowski/}}


\maketitle

\begin{abstract}
  We consider two Large Eddy Simulation (LES) models for the
  approximation of large scales of the equations of
  Magnetohydrodynamics (MHD in the sequel). We study two
  $\alpha$-models, which are obtained adapting to the MHD the approach
  by Stolz and Adams with van Cittert approximate deconvolution
  operators.
  First, we prove existence and uniqueness of a regular weak solution
  for a system with filtering and deconvolution in both
  equations. Then we study the behavior of solutions as the
  deconvolution parameter goes to infinity. The main result of this
  paper is the convergence to a solution of the filtered MHD
  equations. In the final section we study also the problem with
  filtering acting only on the velocity equation.
\end{abstract}

MCS Classification : 76W05, 35Q30, 76F65, 76D03

\medskip

\section{Introduction}
In this paper we study the equations of (double viscous)
incompressible MHD
\begin{equation}
  \label{eq:MHD}
 \begin{aligned}
   &\partial_t\bu+\nabla\cdot(\bu\otimes \bu)-\nabla\cdot(\bB\otimes
   \bB)+\nabla p-\nu\Delta\bu=  \bff,
   \\
   &\partial_t\bB+\nabla\cdot(\bB\otimes \bu)-\nabla\cdot(\bu\otimes
   \bB)- \mu\Delta\bB=0,
   \\
   &\nabla\cdot\bu=\nabla\cdot\bB=0,
   \\
   & \vit (0, \x) = \vit_0 (\x), \qquad \bB (0, \x) = \bB_0 (\x),
 \end{aligned}
\end{equation}
where $\nu>0$ is the kinematic viscosity, while $\mu>0$ is the
magnetic diffusivity.  The fields $\bu$ and $\bB$ are the velocity and
the magnetic field respectively, while the scalar $p$ is the pressure
(rescaled by the density supposed constant here). We consider the
problem in the three dimensional setting, and most of the technical
difficulties are those known for the 3D~Navier--Stokes equations
(NSE). Examples of fluids which can be described by these
equations~\eqref{eq:MHD} are for instance plasmas, liquid metals, and
salt water or electrolytes.  See Davidson~\cite{Dav2001} for an
introduction to the topic. In this paper we aim to study the
approximate deconvolution procedure (developed for turbulent flows by
Stolz and Adams~\cite{SA1999,SAK2001,AS2001}) and especially its
adaption to the MHD with the perspective of numerical simulations of
turbulent incompressible flows, when  coupled to a magnetic field.

In the recent years, the topic of MHD attracted the interest of many
researchers and, for the study of the question of existence,
uniqueness, regularity, and estimates on the number of degrees of
freedom, we recall the following
papers~\cite{Cat2011b,CS2011,HLT2010,LT2010,LT2011,LaTi2011,LST2010,LT2007}.

Approximate Deconvolution Models (ADM) for turbulent flows without
magnetic effects were studied in~\cite{BL2011, DE2006,
  LR2012,Lew2009}. The problem of the limiting behavior of the models
when the grid mesh size goes to zero is already under
control~\cite{DE2006, LL2008, LL06b, RS2010}.On the other hand, the
question of the limiting behavior of the solutions when the
deconvolution parameter goes to infinity is a very recent topic, and
is well-studied just for the NSE --without any coupling--
in~\cite{BL2011} (see also a short review in~\cite{Ber2012}).

In the context of MHD, the topic seems not explored yet, hence we
adapt here the results of~\cite{BL2011} to the equations with the
magnetic field and we find also some interesting unexpected variant,
related to the applications of two different filters. Especially the
equation for the magnetic field turns out to behave much better than
that for the velocity, hence it seems not to require filtering.

To briefly introduce the problem (the reader can find more details in
the introduction of \cite{BL2011}), we recall that the main underlying
idea of LES, see \cite{BIL2006, CLM2012, PS01}, is that of computing
the ``mean values'' of the flow fields $\vit=(u^1,u^2,u^3)$,
$\bB=(B^1,B^2,B^3)$, and $p$. In the spirit of the work started with
Boussinesq \cite{JB77} and then with Reynolds \cite{OR83}, this
corresponds to find a suitable computational decomposition
\begin{equation*}
 \bu=\obu+\bu',\qquad\bB=\overline{\bB}+\bB',\qquad\text{and}\qquad
 p=\overline{p}+p', 
\end{equation*}
where the primed variables are fluctuations around the over-lined mean
fields.  In our context, the mean fields are defined by application of
the inverse of a differential operator.  By assuming that the
averaging operation commutes with differential operators, one gets the
filtered MHD equations
\begin{equation}
  \label{eq:MHD2}
  \begin{aligned}
    \partial_t\moy+\nabla\cdot(\overline {\bu\otimes
      \bu})-\nabla\cdot\overline{(\bB\otimes
      \bB)}+\nabla \overline p -\nu\Delta\moy=\overline \bff,
   \\
    \partial_t\moyb+\nabla\cdot\overline{(\bB\otimes \bu)} -\nabla\cdot\overline{(\bu\otimes
      \bB)}- \mu\Delta\moyb=0,
    \\
    \nabla\cdot\moy=   \nabla\cdot\moyb=0,
    \\
    \moy (0, \x) = \overline{\vit_0} (\x), \qquad \moyb (0, \x) = \overline{\bB}_0 (\x).
  \end{aligned}
\end{equation}
This raises the question of the \textit{interior closure problem},
that is the modeling of the tensors
\begin{equation*}
  \overline{(\bc\otimes\bd)}\qquad
  \text{with either} \mathbf{c},\,\mathbf{d}=\vit \text{ or } \bB
\end{equation*}
in terms of the filtered variables $(\moy, \moyb,\overline{p})$.

From this point, there are many modeling options. The basic model is
the sub-grid model (SGM) that introduces an eddy viscosity of the form
$\nu_t = C h(x)^2 | \nabla \bu|$, which may be deduced from Kolmogorov
similarity theory (see \cite{CLM2012}), where $h(x)$ denotes the local
size of a computational grid, and $C$ is a constant to be fixed from
experiments. This model, that already appears in Prandtl
work~\cite{LP52} with the mixing length $\ell$ instead of $h(x)$, was
firstly used by Smagorinsky for numerical
simulations~\cite{JS63}. This is a very good model, but introduces
numerical instabilities in high-gradient regions, depending on the
numerical scheme and potential CFL constraints.

Among all procedures to stabilize the SGM, the most popular was
suggested by Bardina \textit{et al.}~\cite{BFR80}, which reveals being
a little bit too diffusive and underestimates some of the resolved
scales, that are called ``Sub Filter Scales'' (SFS) (see for
instance~\cite{CSXF05, GC03}). Then the model needs to be
``deconvolved'' to reconstruct accurately the SFS. Hence, many options
occur here, too.  In the present paper we study the Approximate
Deconvolution Model (ADM), introduced by Adams and
Stolz~\cite{SA1999,AS2001}, who have successfully transferred image
modeling procedures~\cite{BB98} to turbulence modeling.

From a simplified and naive mathematical viewpoint, this model, which uses similarity properties of turbulence, is defined by
approximating the filtered bi-linear terms as follows:
\begin{equation*}
  \overline{(\bc\otimes\bd)}\sim
  \overline{(D_N(\overline{\bc})\otimes D_N(\overline{\bd})}\, .
\end{equation*}
Here the \textit{filtering operators} $G_i$ are defined thanks to the
Helmholtz filter (cf.~(\ref{eq:differential_filter})--(\ref{eq:TTTOB})
below) by $G_1(\bu)=\moy$, $G_2(\bB)=\moyb$, where
$G_i:=(\mathrm{I}-\alpha_i^2 \Delta)^{-1}$, $i=1, 2$.  Observe that we
can then have two different filters corresponding to the equation for
the velocity and for that of the magnetic field.  There are two
interesting values for the couple of parameters
$(\alpha_1,\alpha_2)\in\R^+\times\R^+$:
\begin{enumerate}
\item $\alpha_1=\alpha_2>0$. In this case the approximate equations
  conserve Alfv\'en waves, see~\cite{LT2011};
\item $\alpha_1>0$, $\alpha_2=0$, which means no filtering in the
  equation for $\bB$.
\end{enumerate}
The deconvolution operators $D_{N_i}$ are defined through the van Cittert
algorithm~\eqref{eq:deconvolution-operator}
%
and the initial value problem that we consider in the space periodic
setting is:
\begin{equation}
  \label{eq:adm1}
  \begin{aligned}
    & \partial_t \bw +\nabla\cdot G_1 \big(D_{N_1}(\bw) \otimes
    D_{N_1}(\bw)\big) - \nabla\cdot G_1\big(D_{N_2}(\bb) \otimes
    D_{N_2}(\bb)\big) +\nabla q-\nu\Delta\bw=G_1\bff,
    \\
    & \partial_t \bb +\nabla\cdot G_2 \big(D_{N_1}(\bw) \otimes
    D_{N_2}(\bb)\big) - \nabla\cdot G_2\big(D_{N_2}(\bb) \otimes
    D_{N_1}(\bw)\big)- \mu\Delta\bb=0,
    \\
    &\nabla\cdot\bw = \nabla\cdot\bb =0,
    \\
    &\wit (0, \x) = G_1{\vit_0} (\x), \qquad \bb (0, \x) = G_2
    \bB_0(\x),
    \\
    &\alpha_1 >0, \qquad \alpha_2 \geq 0.
  \end{aligned}
\end{equation}
As usual, we observe that the equations~\eqref{eq:adm1} are not the
equations~\eqref{eq:MHD2} satisfied by $(\obu,\overline{\bB})$, but we
are aimed at considering~\eqref{eq:adm1} as an approximation
of~\eqref{eq:MHD2}, hence $\bw\simeq G_1\bu$ and $\bb\simeq
G_2\bB$. This is mathematically sound since, at least
\textit{formally,}
\begin{equation*}
  D_{N_i}\to A_i:=\mathrm{I}-\alpha^2_i \Delta \quad \text{ in
    the limit }N_i \to+\infty ,
\end{equation*}
hence, again formally,~\eqref{eq:adm1} will become the filtered MHD
equations~\eqref{eq:MHD2}.  The existence and uniqueness issues have
been also treated (even if without looking for estimates independent
of $N_i$) in~\cite{LT2011,LT2010} (for arbitrary deconvolution
orders). What seems more challenging is to understand whether this
convergence property is true or not, namely to show that as the
approximation parameters $N_i$ grow, then (as recently proved for the
Navier--Stokes equations in~\cite{BL2011})
\begin{equation*}
  \bw\to G_1\bu,\qquad \bb\to G_2\bB,\qquad\text{and}\qquad q\to
  G_1{q}.
\end{equation*}
%
We prove that the model~(\ref{eq:adm1}) converges, in some sense, to
the averaged MHD equations~\eqref{eq:MHD2}, when the typical scales of
filtration $\alpha_i$ remain fixed. Before analyzing such a convergence,
we need to prove more precise existence results. To this end we follow
the same approach from~\cite{BL2011}, which revisits the approach
in~\cite{DE2006} for the Navier--Stokes equations.  To be more
precise, the main result deals with $\alpha_1>0$ and $\alpha_2>0$. We
first prove (cf.~Thm.~\ref{thm:existence}) existence and uniqueness of
solutions $(\wit_N, \bb_B,q_N)$ of \eqref{eq:adm1}, with
$N=(N_1,N_2)$, such that
\begin{equation*}
 \begin{aligned}
   &\wit_N,\,\bb_N\in L^2 ([0,T]; H^2(\tore)^3) \cap L^\infty ([0,T];
   H^1(\tore)^3),
   \\
   &q_N\in L^2([0,T];W^{1,2}(\tore))\cap L^{5/3}([0,T];  W^{2, 5/3}(\tore)),
 \end{aligned}
\end{equation*}
and our main result is the following one.
\begin{theorem}
  \label{thm:Principal}
  Let $\alpha_1>0$ and $\alpha_2>0$; then, from the sequence
  $\{(\wit_N, \bb_N,q_N)\}_{N \in \N^2}$, one can extract a (diagonal)
  sub-sequence (still denoted $\{(\wit_N, \bb_N,q_N)\}_{N \in \N^2}$)
\begin{equation*}
   \begin{aligned}
     &
       \begin{aligned}
         \bw_N\to\bw
         \\
         \bb_N\to\bb
       \end{aligned}
     \quad&
     \left\{     \begin{aligned}
       &\text{weakly in }L^2 ([0,T]; H^2(\tore)^3)
       \\
        &\text{weakly$\ast$ in }L^\infty
       ([0,T]; H^1(\tore)^3),
     \end{aligned}\right.
     \\
     &
       \begin{aligned}
         \bw_N\to\bw
         \\
         \bb_N\to\bb
       \end{aligned}
     &\text{strongly in } L^p ([0,T];
     H^1(\tore)^3),\quad\forall \,1\leq p<+\infty,\qquad
     \\
     &\quad q_N\to q&\text{weakly in }L^2([0,T];W^{1,2}(\tore))\cap
     L^{5/3}([0,T]; W^{2,5/3}(\tore)),
   \end{aligned}
 \end{equation*}
such that the system
 \begin{equation}
   \label{eq:adm2}
     \begin{aligned}
    \partial_t\bw+\nabla\cdot G_1 ({A_1\bw\otimes
      A_1\bw})-\nabla\cdot G_1{(A_2\bb\otimes
      A_2\bb)}+\nabla  q -\nu\Delta\bw= G_1 \bff,
    \\
    \nabla\cdot\bw=   \nabla\cdot\bb=0,
    \\
    \partial_t\bb+\nabla\cdot G_2{(A_2\bb\otimes A_1\bw)}-\nabla\cdot
    G_2{(A_1\bw\otimes A_2\bb)}-\mu\Delta\bb=0,
    \\
    \bw (0, \x) = G_1{\vit_0} (\x), \qquad \bb (0, \x) = G_2{\bB}_0
    (\x),
  \end{aligned}
 \end{equation}
 holds in the distributional sense.  Moreover, the following energy
 inequality holds:
 \begin{equation}
   \label{eq:energy-inequality}
   \frac{1}{2}\frac{d}{dt}\big(\|A_1\wit\|^2+\|A_2\bb \|^2\big)+
   \nu\|\g A_1\wit\|^2+\mu\|\g A_2\bb\|^2\le\langle{\bf f}, A_1\wit\rangle.
 \end{equation}
\end{theorem}
As a consequence of Theorem \ref{thm:Principal}, we deduce that the
field $(\vit, \bB,p)=(A_1\wit, A_2\bb,A_1 q)$ is a dissipative (of
Leray-Hopf's type) solution to the MHD Equations~\eqref{eq:MHD}.
\begin{remark} 
  Following the work carried out in~\cite{DL2012} about ADM without
  coupling, we conjecture that the error modeling in the case of the
  approximate deconvolution MHD is of order $N^{-1/2}$.
\end{remark}
\begin{remark} 
  The question of the boundary conditions is the great challenge in
  LES modeling, see~\cite{PB02} for a general review.  This is why,
  either from physical or mathematical viewpoint, theoretical issues
  are raised in the case of periodic boundary conditions, although the
  reality of such boundary conditions may be controversial.  In a
  paragraph in~\cite{SAK2001}, the authors outline a possible
  numerical discrete algorithm for a deconvolution procedure, by the
  finite difference method. This is proposed in the case of
  homogeneous boundary conditions, but there is no mathematical
  analysis about this method, which still remains an open problem.
\end{remark}
%
%
\bigskip

\noindent\textbf{Plan of the paper.}
In Sec.~\ref{sec:background} we introduce the notation and the
filtering operations. In Sec.~\ref{sec:existence}-\ref{sec:limit}
we consider the model with the double filtering with non-vanishing
parameters $\alpha_i$ and then we study the limiting behavior as
$N_i\to+\infty$.  Finally, in Section~\ref{sec:existence2}, we treat
the same problems in the case $\alpha_1>0$ and $\alpha_2=0$.  Since
most of the calculations are in the same spirit of those
in~\cite{BL2011}, instead of proofs at full length we just point out
the changes needed to adapt the proof valid for the NSE to the MHD
equations.
\section{Notation and Filter/Deconvolution operators}
\label{sec:background}%
This section is devoted to the definition of the functional setting
which we will use, and to the definition of the filter through the
Helmholtz equation, with the related deconvolution operator. All the
results are well-known and we refer
to~\cite{BL2011,LR2012,Lew2009} for further details.
%
%
We will use the customary Lebesgue $L^p$ and Sobolev $W^{k,p}$ and
$W^{s,2}=H^s$ spaces, in the periodic setting. Hence, we use Fourier
series on the 3D torus $\tore$.  Let be given
$L\in\R^\star_+:=\{x\in\R:\ x>0\}$, and define
$\Omega:=]0,L[^3\subset\R^3$. We denote by $({\bf e}_1, {\bf e}_2,
{\bf e}_3)$ the orthonormal basis of $\R^3$, and by $\x:=(x_1, x_2,
x_3) \in \R^3$ the standard point in $\R^3$. We put ${\cal T}_3 := 2
\pi \Z^3 /L $ and $\mathbb{T}_3$ is the torus defined by
$\mathbb{T}_3:=\left( \R^3 / {\cal T}_3\right)$.  We use $\|\cdot\|$
to denote the $L^{2}(\mathbb{T}_3)$-norm and associated operator
norms. We always impose the zero mean condition on the fields that we
consider %
and we define, for a general exponent $s\geq0$,
\begin{equation*}
 \mathbf{H}_{s} := \left\{\wit : \tore \rightarrow \R^3, \, \,
   \mathbf{w} \in H^{s}(\mathbb{T}_3)^3,   \quad\nabla\cdot\mathbf{w}
   = 0, \quad\int_{\mathbb{T}_3}\mathbf{w}\,d\x = \mathbf{0} \right\}.
\end{equation*}
%
For $\wit \in {\bf H}_s$, we can expand the fields as $ \wit
(\mathbf{x})=\sum_{\mathbf{k} \in {\cal T}_3^\star}\widehat{\wit}_{\bf
  k} e^{+i\mathbf{k \cdot x}}$, where $\mathbf{k}\in{\cal T}_3^\star${
  is the wave-number,} and the Fourier coefficients are
$\widehat{\wit}_{\bf
  k}:=\frac{1}{|\mathbb{T}_3|}\int_{\mathbb{T}_3}\wit(\x)e^{-i\mathbf{k
    \cdot x}}d\mathbf{x}$. The magnitude of $\mathbf{k}$ is defined by
$k:=|\mathbf{k}|=\{|k_{1}|^{2}+|k_{2}|^{2}+|k_{3}|^2\}^{\frac{1}{2}}$. We
define the $\mathbf{H}_{s}$ norms by $\| \wit \|^{2}_{s} :=
\sum_{\mathbf{k} \in {\cal T}_3^\star} | \mathbf{k} |^{2s} |\widehat
{\wit }_{\bf k}|^{2}$,
%
where of course $\| \wit \|^{2}_{0} := \| \wit \|^{2}$. The inner
products associated to these norms are
%
%
$  (\wit, \hbox{\bf v} )_{\mathbf{H}_s} := \sum_{\mathbf{k} \in {\cal T}_3^\star}
  | \mathbf{k} |^{2s}  \widehat
  {\wit }_{\bf k}\cdot  \overline{\widehat
    {\hbox{\bf v} }_{\bf k}}$, %
%
%
  where $\overline{\widehat{\hbox{\bf v} }_{\bf k}}$ denotes the
  complex conjugate of $\widehat{\hbox{\bf v} }_{\bf k}$.  To have
  real valued vector fields, we impose
  $\widehat{\wit}_{-\mathbf{k}}=\overline{\widehat{\wit }_{\bf k}}$
  for any $\mathbf{k}\in {\cal T}_3^\star$ and for any field denoted
  by $\wit$.
%
It can be shown (see e.g.~\cite{DG1995}) that when $s$ is an integer,
$\|\wit\|^{2}_{s}:=\| \nabla^{s} \wit\|^{2}$ and also, for general $s
\in \R$, $(\mathbf{H}_s)' = \mathbf{H}_{-s} $.
%
%

We now recall the main properties of the Helmholtz filter. In the
sequel, $\alpha>0$ denotes a given fixed number and for
$\wit\in\mathbf{H}_{s}$ the field $\witt$ is the solution of the
Stokes-like problem:
\begin{equation}\label{eq:differential_filter}
\begin{aligned}
  -\alpha^2\Delta\witt +\witt+\nabla \pi=\wit &\qquad\text{in }\tore,
  \\
  \nabla\cdot \witt=0&\qquad\text{in }\tore,
  \\
\int_{\tore}\pi \,d\x=0.
\end{aligned}
\end{equation}
%
For $\wit\in\mathbf{H}_{s}$ this problem has a unique solution
$(\witt,\pi)\in\mathbf{H}_{s+2}\times H^{s+1} (\tore)$, whose velocity
is denoted also by $\obw=G (\bw)$.
Observe that, with a common abuse of notation, for a scalar function
$\chi$ we still denote (this is a standard notation) by $\overline
\chi$ the solution of the pure Helmholtz problem
\begin{equation}
  \label{eq:TTTOB}
  A \overline \chi :=-\alpha^2 \Delta \overline \chi + \overline \chi =
  \chi \quad \text{in }\tore.
\end{equation}
In particular, in the LES model~\eqref{eq:adm1} and in the filtered
equations~\eqref{eq:MHD2}--\eqref{eq:adm2}, the symbol
``$\overline{\phantom{a\cdot}}$'' denotes the pure Helmholtz filter,
applied component-wise to the various tensor fields.

We recall now a definition that we will use several times in the
sequel.
\begin{definition}
 Let $K$ be an operator acting on ${\bf H}_s$. Assume that $\expk$
 are eigen-vectors of $K$ with corresponding eigenvalues $\widehat
 K_{\bf k}$. Then we shall say that $\widehat K_{\bf k}$ is the
 symbol of $K$.
\end{definition}
The deconvolution operator $D_N$ is constructed thanks to the
Van-Cittert algorithm by
$ D_N := \sum_{n=0}^N (\mathrm{I}-G)^n$.
%
Starting from this formula, we can express the deconvolution operator
in terms of Fourier series
$ D_N (\wit) = \sum_{ {\bf k}\in {\cal T}_3^\star}\widehat D_N ({\bf k})
 \widehat  \wit_{\bf k} e^{+i\mathbf{k \cdot x}}$,
where
\begin{equation}
  \label{eq:rep_D_N}
 \begin{aligned}
   \widehat{D}_N({\bf k})=\sum_{n=0}^N\left(\frac{\alpha^2|{\bf
         k}|^2}{1+\alpha^2|{\bf k}|^2}\right)^n= (1+\alpha ^2 |{\bf
     k}|^2)\rho_{N,{\bf k}} , \qquad \rho_{N,{\bf
       k}}=1-\left(\frac{\alpha^2 |{\bf k}|^2}{1+\alpha^2 |{\bf
         k}|^2}\right)^{N+1}.
 \end{aligned}
\end{equation}
%
The basic properties satisfied by $\widehat{D}_N$ that we will need
are summarized in the following lemma.
\begin{lemma}
  \label{lem:lower_bound}
 For each $N\in\N$ the operator $\bn:\ \mathbf{H}_{s} \to
 \mathbf{H}_{s}$ is self-adjoint, it commutes
 with differentiation, and the following properties hold true:
 \begin{eqnarray}
   & \label{eq:IINVC9} 1\leq \widehat{D}_N({\bf k})\leq N+1   & \forall\, {\bf k} \in {\cal T}_3;
   \\
   & \label{JJV34}  \displaystyle \widehat D_N ({\bf k}) \approx (N+1)
   {1+\alpha^2 |{\bf k}|^2 \over \alpha^2 |{\bf k}|^2 } & \hbox{for
     large } |{\bf k}|;
   \\
   &  \label{TT0ON}  \displaystyle  \lim_{|{\bf k}
     |\to+\infty}\widehat{D}_N({\bf k})=N+1  & \text{for fixed    }\alpha>0;
   \\
   &
   \widehat{D}_N({\bf k})\leq(1+\alpha^2|{\bf k}|^2)  &  \forall\,
   {\bf k} \in {\cal T}_3,\     \alpha>0;
   \\
   &\text{ the map }    \wit \mapsto D_N (\wit) \text{
     is an isomorphism}& \text{s.t. }
   \| D_N \|_{\mathbf{H}_s} = O(N+1)\quad \forall\,s \ge 0;
   \\
   &\label{lem:converges} \displaystyle
   \lim_{N\to+\infty} \bn(\bw)=A\bw \text{ in }\bH_s&\forall\, s \in \R \text{ and }\bw\in \bH_{s+2}.
\end{eqnarray}
\end{lemma}
All these claims follow from direct inspection of the
formula~\eqref{eq:rep_D_N} and, in the sequel, we will also use the
natural notations
$G_i:=A_i^{-1}:=(\mathrm{I}-\alpha_i^2\Delta)^{-1}$ and
\begin{equation}
  \label{eq:deconvolution-operator}
   D_{N_i} := \sum_{n=0}^{N_i} (\mathrm{I}-G_i)^n,\qquad i=1,2.
\end{equation}
%
%
\section{Existence results} \label{sec:existence}
In order to be self-contained, we start by considering the initial
value problem for the model~\eqref{eq:adm1}. In this section,
$N_1,\,N_2\in\N $ are fixed as well as $\alpha_1>0$, $\alpha_2>0$, and
we assume that the data are such that
\begin{equation}
  \label{eq:reg-initial-data}
  \vit_0, \bB_0 \in \mathbf{H}_{0}, \quad {\bf f}\in  L^2([0,T]\times\tore),
\end{equation}
which naturally yields $ G_1{\vit_0}, G_2{\bB_0} \in \mathbf{H}_{2},
\; G_1{\bf f}\in L^2([0,T]; \mathbf{H}_{2})$.  We start by defining
the notion of what we call a ``regular weak'' solution to this system.
\begin{definition}[``Regular weak'' solution]
  \label{RegSol}
 We say that the triple $(\wit, \bb, q)$ is a ``regular weak'' solution to
 system~(\ref{eq:adm1}) if and only if the three following items are
 satisfied: \smallskip

\noindent \textcolor{red}{1) \underline{\sc Regularity:} }
 \begin{eqnarray}
   && \label{eq:Reg11} \wit, \bb \in   L^2([0,T];
   \mathbf{H}_{2}) \cap C([0,T];  \mathbf{H}_{1}),
   \\
   && \label{eq:Reg12} \p_t \wit, \p_t \bb \in L^2([0,T];  \mathbf{H}_{0})
   \\
   && \label{eq:Reg13} q \in L^2([0,T]; H^1(\tore)),
 \end{eqnarray}
 \noindent\textcolor{red}{2) \underline {\sc Initial data:} }
 \begin{equation}
   \label{eq:Initial}
   \displaystyle \lim_{t \rightarrow 0}\|\wit(t, \cdot) -  G_1{\vit_0}  \|_{ \mathbf{H}_{1}} = 0, \qquad
   \lim_{t \rightarrow 0}\|\bb(t, \cdot) -  G_2{\bB_0}  \|_{ \mathbf{H}_{1}} = 0,
 \end{equation}

 \noindent\textcolor{red}{3) \underline{\sc Weak Formulation:}} For all $\vittest, \bh  \in L^2([0,T];H^1(\tore)^3)$
 \begin{eqnarray}
   &&
\label{eq:weak-formulation1}
   \begin{aligned}
     & \displaystyle\int_0^T\int_{\tore} \p_t \wit \cdot \vittest -
     \int_0^T\int_{\tore} G_1({D_{N_1} (\wit) \otimes D_{N_1}
       (\wit)}):\g\vittest
     \\
     & \qquad + \int_0^T\int_{\tore} G_1({D_{N_2} (\bb) \otimes
       D_{N_2} (\bb)}):\g\vittest +\int_0^T\int_{\tore} \g q \cdot
     \vittest
     \\
     &
     \qquad+\nu\int_0^T\int_{\tore}\g\wit:\g\vittest=\int_0^T\int_{\tore}
     G_1({\bf f})\cdot\vittest,
   \end{aligned}
   \\
      &&
  \label{eq:weak-formulation2}
      \begin{aligned}
        & \displaystyle\int_0^T\int_{\tore} \p_t \bb \cdot \bh -
        \int_0^T\int_{\tore} G_2({D_{N_2} (\bb) \otimes D_{N_1}
          (\bw)}) : \g \bh
        \\
        & \qquad + \int_0^T\int_{\tore} G_2({D_{N_1} (\bw) \otimes
          D_{N_2} (\bb)}) : \g \bh+ \mu \int_0^T\int_{\tore} \g \bb :
        \g \bh=0.
   \end{aligned}
 \end{eqnarray}
\end{definition}
Observe that, for simplicity, we suppressed all $d\bx$ and $dt$ from
the space-time integrals. With the same observations as
in~\cite{BL2011}, one can easily check that all integrals involving
$D_{N_1}\wit$ and $D_{N_2}\bb$
in~\eqref{eq:weak-formulation1}--\eqref{eq:weak-formulation2} are
finite under the regularity in \eqref{eq:Reg11}-\eqref{eq:Reg12}.
We now prove the following theorem, which is an adaption of the
existence theorem in~\cite{BL2011} and at the same time a slightly
more precise form of the various existence theorems available in
literature for doubly viscous MHD systems.
\begin{theorem}
  \label{thm:existence}
  Assume that~(\ref{eq:reg-initial-data}) holds, $0<\alpha_i \in\R$
  and $N_i \in \N$, $i=1,2$, are given and fixed. Then,
  problem~\eqref{eq:adm1} has a unique regular weak solution.
\end{theorem}
In the proof we use the usual Galerkin method (see for instance the
basics for incompressible fluids in~\cite{Lio1969}) with
divergence-free finite dimensional approximate velocities and magnetic
fields. 
We also point out that Theorem~\ref{thm:existence} greatly improves
the corresponding existence result in~\cite{LT2011} and it is not a
simple restatement of those results. Some of the main original
contributions are here the estimates, uniform in $N$, that will allow
later on to pass to the limit when $N_i\to+\infty$.
%
%
%
%
  \begin{proof}[Proof of Theorem~\ref{thm:existence}]
    Let be given $\ind \in \N^\star$ and define ${\bf V}_\ind$ to be
    the following space of real valued trigonometric polynomial vector
    fields
    \begin{equation*}
      {\bf V}_\ind := \{ \wit \in  \mathbf{H}_{1}:\quad
      \int_{\tore}\wit (\x) \, \expk= {\bf 0},\quad \forall \, {\bf
        k} \,    \text{ with }\,
      |{\bf k}|>\ind \}.
  \end{equation*}
%
%
  In order to use classical tools for systems of ordinary differential
  equations, we approximate the external force $\bff$ with
  $\bff_{1/m}$ by means of Friederichs mollifiers.  Thanks to the
  Cauchy-Lipschitz Theorem, we can prove existence of $T_\ind>0$ and
  of unique $C^1$ solutions $ \wit_\ind(t,{\bf x})$ and
  $\bb_\ind(t,{\bf x})$ (belonging to $ {\bf V}_\ind$ for all
  $t\in[0,T_\ind[$) to
  \begin{eqnarray}
    \label{eq:RQPJ89}
   &&
   \begin{aligned}
     & \displaystyle\int_{\tore} \p_t \wit_\ind \cdot \vittest -
     \int_{\tore} G_1({D_{N_1} (\wit_\ind) \otimes D_{N_1} (\wit_\ind)}) :
     \g  \vittest
     \\
     & \qquad + \int_{\tore} G_1({D_{N_2} (\bb_\ind) \otimes D_{N_2} (\bb_\ind)}) : \g  \vittest
     \\
     & \qquad      + \nu \int_{\tore} \g \wit_\ind : \g  \vittest =\int_{\tore} G_1({\bf f}_{1/m}) \cdot
     \vittest,
   \end{aligned}
   \\
   &&
   \begin{aligned}
     & \displaystyle\int_{\tore} \p_t \bb_\ind \cdot \bh -
     \int_{\tore} G_2({D_{N_2} (\bb_\ind) \otimes D_{N_1} (\bw_\ind)}) : \g
     \bh
     \\
     & \qquad + \int_{\tore} G_2({D_{N_1} (\bw_\ind) \otimes D_{N_2} (\bb_\ind)}) : \g  \bh
 + \mu \int_{\tore} \g \bb_\ind : \g  \bh=0\,,
   \end{aligned}
 \end{eqnarray}
 for all $\vittest, \bh \in L^2([0,T];\mathbf{V}_m)$.
 \begin{remark}
   Instead of $(\bw_\ind,\bb_\ind)$,  a more precise and appropriate
   notation for the solution of the Galerkin system would be
  %
$ (      \bw_{\ind,N_1,N_2,\alpha_1,\alpha_2},
\bb_{\ind,N_1,N_2,\alpha_1,\alpha_2})$. %
%
We are asking for a simplification, since in this section $N_i$ and
$\alpha_i$ are fixed and the only relevant parameter is the Galerkin
one $m\in\N^\star$.
  \end{remark}
%
%
 %
  The natural and correct test functions to get \textit{a priori}
  estimates are $A_1D_{N_1} (\wit_\ind)$ for the first equation and
  $A_2 D_{N_2}(\bb_\ind)$ for the second one. Arguing as in \cite{BL2011}, it is easily checked that both are in $V_m$. Since $A_1, A_2$ are
  self-adjoint and commute with differential operators, it holds:
  \begin{equation*}
    \begin{aligned}
      &\int_{\tore} G_1({D_{N_1} (\wit_\ind)
        \otimes D_{N_1} (\wit_\ind)})
      : \g  (A_1 D_{N_1} (\wit_\ind))
      \,d\x=0,
      \\
      &\int_{\tore} G_2({D_{N_2} (\bb_\ind)
        \otimes D_{N_2} (\bb_\ind)})
      : \g  (A_2 D_{N_2} (\bb_\ind))
      \,d\x=0.
    \end{aligned}
  \end{equation*}
%
  Moreover,
  \begin{align*}
    &\int_{\tore} G_1 (D_{N_2}(\bb_\ind)\otimes D_{N_2}(\bb_\ind)):\g (A_1
    D_{N_1}(\bw_\ind)) \, d\x
    \\
    & \qquad -\int_{\tore} G_2(D_{N_2}(\bb_\ind)\otimes D_{N_1}(\bw_\ind)):\g
    (A_2 D_{N_2}(\bb_\ind)) \, d\x
    \\
    &\qquad +\int_{\tore} G_2(D_{N_1}(\bw_\ind)\otimes D_{N_2}(\bb_\ind)):\g
    (A_2 D_{N_2}(\bb_\ind)) \, d\x
    \\
    & = -\int_{\tore} (D_{N_2}(\bb_\ind) \cdot \nabla) D_{N_2}(\bb_\ind) \cdot
    D_{N_1}(\bw_\ind) \, d\x
    \\
    & \qquad + \int_{\tore} (D_{N_1}(\bw_\ind) \cdot \nabla) D_{N_2}(\bb_\ind)
    \cdot D_{N_2}(\bb_\ind) \, d\x
    \\
    & \qquad - \int_{\tore} (D_{N_2}(\bb_\ind) \cdot \nabla) D_{N_1}(\bw_\ind)
    \cdot D_{N_2}(\bb_\ind) \, d\x = 0\, .
  \end{align*}
  Summing up the equations satisfied by $\wit_\ind$ and $\bb_\ind$,
 %
%
  using standard integration by parts and
  Poincar\'e's inequality combined with Young's inequality, we obtain
  \begin{equation}\label{HHW12}
    \begin{aligned}
      &\|A_1^{\frac{1}{2}} D_{N_1}^{\frac{1}{2}}
      (\wit_\ind)(t, \cdot) \|^2 + \|A_2^{\frac{1}{2}}
      D_{N_2}^{\frac{1}{2}} (\bb_\ind)(t, \cdot) \|^2
      \\
      & \qquad+\int_0^t\|\g A_1^{\frac{1}{2}}D_{N_1}^{\frac{1}{2}}(\wit_\ind)\|^2\,d\tau
      +\int_0^t\|\g A_2^{\frac{1}{2}}D_{N_2}^{\frac{1}{2}} (\bb_\ind)\|^2\,d\tau
      \\
      &\qquad \le C ( \| \vit_0 \|, \, \| \bB_0\|, \, \nu^{-1} \| {\bf f}
      \|_{L^2 ([0,T]; \mathbf{H}_{-1}) }),
  \end{aligned}
  \end{equation}
  which shows that the natural quantities under control are
  $A_1^{\frac{1}{2}}D_{N_1}^{\frac{1}{2}} (\wit_\ind)$ and
  $A_2^{\frac{1}{2}}D_{N_2}^{\frac{1}{2}} (\bb_\ind)$.

  Since we need to prove many \textit{a priori} estimates, for the
  reader's convenience we organize the results in tables
  as~\eqref{eq:Estimates}. In the first column we have labeled the
  estimates, while the second column specifies the variable under
  concern. The third one explains the bound in terms of function
  spaces: The symbol of a space means that the considered sequence is
  bounded in such a space. %
  Finally, the fourth column states the order in terms of $\alpha$,
  $\ind$ and $N$ for each bound.
%
  \begin{equation}
    \label{eq:Estimates}
    \begin{array} {  @{}>{\vline \hfill  } c @{\,
          \, \,  }>{\vline  \hfill} c @{\, \, \,  }>{\vline \hfill } c
        @{\, \, \,  }>{\vline \, \, \,  } c   @{\hfill \vline  }}
      \hline
      \,\,  \hbox{Label} & \hbox{Variable} &\hbox{   bound          }
      \hskip 2cm ~  \hfill &      \hfill\hbox{order} \quad ~
      \\
      \hline
      \,a)\ \  \phantom{} & \, A^{\frac{1}{2}}_1D_{N_1}^{\frac{1}{2}}(
      \wit_\ind),\,A^{\frac{1}{2}}_2D_{N_2}^{\frac{1}{2}} (\bb_\ind)&
      \,L^\infty([0,T];\mathbf{H}_{0})\cap L^2([0,T];\mathbf{H}_{1})   &
      \hspace{.2cm} O(1)
      \\
      \hline b)\ \  \phantom{}& D_{N_1}^{1/2}(\wit_\ind), \,
      D_{N_2}^{1/2}(\bb_\ind)\, \,  &
      L^\infty([0,T];\mathbf{H}_{0}) \cap L^2([0,T];  \mathbf{H}_{1})   & \hspace{.2cm} O(1)
      \\
      \hline c )\ \  \phantom{}& D_{N_1}^{1/2}(\wit_\ind), \,
      D_{N_2}^{1/2}(\bb_\ind) \, \,  &
      L^\infty([0,T];\mathbf{H}_{1})\cap L^2([0,T];  \mathbf{H}_{2})   & \hspace{.2cm}
      O(\alpha^{-1})
      \\
      \hline d)\ \  \phantom{} & \wit_\ind, \, \bb_\ind\, \, \, \, \,
      &
      L^\infty([0,T];\mathbf{H}_{0})\cap L^2([0,T];\mathbf{H}_{1})   &\hspace{.2cm} O(1)
      \\
      \hline
      e )\ \  \phantom{}& \wit_\ind, \, \bb_\ind\, \, \, \, \,  &
      L^\infty([0,T];\mathbf{H}_{1})\cap L^2([0,T];\mathbf{H}_{2})   &
      \hspace{.2cm} O(\alpha^{-1})
      \\
      \hline
      f)\ \  \phantom{} & D_{N_1} (\wit_\ind), \, D_{N_2} (\bb_\ind)\,
      \, & L^\infty([0,T];\mathbf{H}_{0})\cap
      L^2([0,T];\mathbf{H}_{1})   & \hspace{.2cm} O(1)
      \\
      \hline
      g) \raisebox{3ex} \ \  \phantom{} & D_{N_1} (\wit_\ind), \, D_{N_2} (\bb_\ind)\,
      \, \, & L^\infty([0,T];\mathbf{H}_{1})\cap
      L^2([0,T];\mathbf{H}_{2})   &  O(\frac{ \sqrt{N_i+1}}{\alpha}) \,~
      \\[.5ex]
      \hline
      h)\ \  \phantom{} &\p_t \wit_\ind, \,\p_t \bb_\ind  \, \, &
      L^2([0,T];\mathbf{H}_{0})   \hspace{.2cm} ~  &  \hspace{.2cm} O(\alpha^{-1}).
      \\
      \hline
    \end{array}
  \end{equation}
  In the previous table, $\alpha=\alpha_1$ for $\bw_\ind$,
  $\alpha=\alpha_2$ for $\bb_\ind$, while in $h)$ we can take
  $\alpha:=\min\{\alpha_1, \alpha_2\}$ for both $\bw_\ind$ and
  $\bb_\ind$.

  \medskip

  {\sl Proof of~(\ref{eq:Estimates}-a) } --- This estimate follows
  directly from~\eqref{HHW12}.  Notice also that since the operator
  $A_{i}^{-{\frac{1}{2}}}D_{N_i}^{\frac{1}{2}}$ has for symbol
  $\rho_{N_i, {\bf k}}^{1/2} \le 1$, then
  $\|A_1^{\frac{1}{2}}D_{N_1}^{\frac{1}{2}}G_1{\bf f}_{1/m}\|\le
  C\|{\bf f} \|$ and also
  \begin{equation*}
    \|A_{i}^{\frac{1}{2}}D_{N_i}^{\frac{1}{2}}\mathbb{P}_{\ind} G_i{\mathbf{a}}\|=
    \|\mathbb{P}_{\ind}A_{i}^{\frac{1}{2}}D_{N_i}^{\frac{1}{2}} G_i\mathbf{a}\|\le \|
    A_i^{\frac{1}{2}}D_{N_i}^{\frac{1}{2}}\mathbf{a}\| \le \|
    \mathbf{a}\|,
  \end{equation*}
  which will be used with $\mathbf{a}=\vit_0,\,\bB_0$.

 \medskip

  {\sl Proof of~(\ref{eq:Estimates}-b)-(\ref
    {eq:Estimates}-c)} --- Let $\vittest \in {\bf H}_2 $. Then, with
  obvious notations, one has
  \begin{equation*}
    \|A_i^{\frac{1}{2}} \vittest \|^2 = \sum_{{\bf k} \in {\cal T}_3^\star }
    (1 + \alpha_i^2 | {\bf k} |^2) | \widehat \vittest_{\bf k}|^2 =
    \|\vittest \|^2 + \alpha^2_i \|\g \vittest \|^2.
  \end{equation*}
  It suffices to apply this identity to
  $\vittest=D_{N_1}^{\frac{1}{2}}(\wit_\ind),\,D_{N_2}^{\frac{1}{2}}(\bb_\ind)$
  and to $\vittest=\p_i D_{N_1}^{\frac{1}{2}} (\wit_\ind),\,\p_i
  D_{N_1}^{\frac{1}{2}} (\bb_\ind)$ ($i=1,2,3$) in~(\ref{HHW12}) to
  get the claimed result.

  \medskip

  {\sl Proof
    of~(\ref{eq:Estimates}-d)-(\ref{eq:Estimates}-e)-(\ref{eq:Estimates}-f)}
  ---These are  direct consequence
  of~~(\ref{eq:Estimates}-a)-(\ref{eq:Estimates}-b)-(\ref{eq:Estimates}-c) combined
  with~(\ref{eq:IINVC9}).

\medskip

  {\sl Proof of~(\ref{eq:Estimates}-g)} --- This follows
  directly from~(\ref{eq:Estimates}-e), together
  with~(\ref{eq:IINVC9}).
 \begin{remark}
    One crucial point is that~(\ref{eq:Estimates}-g) is valid for each
    $N=(N_1,N_2)$, but the bound may grow with $N_i$.
  \end{remark}

  \medskip

  {\sl Proof of~(\ref{eq:Estimates}-h)} --Let us take
  $\p_t \wit_\ind , \p_t \bb_\ind\in {\bf V}_\ind$ as test vector fields
  in~(\ref{eq:RQPJ89}). We get
  \begin{align*}
    & \|\p_t \wit_\ind \|^2 + \int_{\tore} G_1 \big( \nabla \cdot
    [D_{N_1} (\bw_\ind) \otimes D_{N_1} (\bw_\ind) ] \big) \cdot \p_t
    \wit_\ind
    \\
    & \quad - \int_{\tore} G_1 \big( \nabla \cdot [D_{N_2} (\bb_\ind)
    \otimes D_{N_2} (\bb_\ind) ] \big) \cdot \p_t \wit_\ind +
    \frac{\nu}{2}\frac{d}{dt} \| \g \wit_\ind \|^2 = \int_{\tore}
    G_1{\bf f}_{1/\ind} \cdot \p_t \wit_\ind,
\\
& \|\p_t \bb_\ind \|^2
    + \int_{\tore} G_2 \big( \nabla \cdot [D_{N_2} (\bb_\ind) \otimes
    D_{N_1} (\bw_\ind) ] \big) \cdot \p_t \bb_\ind
    \\
    & \quad - \int_{\tore} G_2 \big( \nabla \cdot [D_{N_1} (\bw_\ind)
    \otimes D_{N_2} (\bb_\ind) ] \big) \cdot \p_t \bb_\ind +
    \frac{\mu}{2}\frac{d}{dt} \| \g \bb_\ind \|^2 = 0\, .
  \end{align*}
  To estimate the time derivative, we need bounds on the bi-linear terms
 \begin{equation*}
   \begin{aligned}
     & {\bf A}_{N,\ind} := G_1 {\g \cdot \big(D_{N_1} (\wit_\ind)
       \otimes D_{N_1}( \wit_\ind)\big)},
     \\
     & {\bf B}_{N,\ind} := G_1 {\g \cdot \big(D_{N_2} (\bb_\ind)
       \otimes D_{N_2}( \bb_\ind)\big)},
     \\
     & {\bf C}_{N,\ind} := G_2 {\g \cdot \big(D_{N_1} (\wit_\ind)
       \otimes D_{N_2}( \bb_\ind)\big)}.
  \end{aligned}
  \end{equation*}
  Even if we have two additional terms, this can be easily done as
  in~\cite{BL2011} by observing that,
by interpolation inequalities, both $D_{N_1}(\wit_\ind)$ and
$D_{N_2}(\bb_\ind) $ belong to $L^4 ([0,T];L^3(\tore)^3)$.
Therefore,
by observing that  the operator $(\g \cdot) \circ G_i$ has
symbol corresponding to the inverse of one space derivative,
it easily follows that
${\bf A}_{N,\ind},{\bf B}_{N,\ind},{\bf C}_{N,\ind} \in L^2 ([0,T]\times \tore)^3$.
Moreover, the
bound is of order $O(\alpha_i^{-1})$ as well.

  \bigskip

%
%
  From the bounds proved in~(\ref{eq:Estimates}) and classical
  Aubin-Lions compactness tools, we can extract sub-sequences
  $\{\wit_\ind,\,\bb_\ind\}_{\ind \in \N}$ converging to $\wit,\bb \in
  L^\infty([0,T];\mathbf{H}_{1}) \cap L^2([0,T];\mathbf{H}_{2})$ and
  such that
  \begin{eqnarray}
    &
    \begin{aligned}
      \wit_\ind\rightarrow\bw
      \\
      \bb_\ind\rightarrow\bb
    \end{aligned}
    &\hbox{ weakly in }
    L^2([0,T];  \mathbf{H}_{2}),
    \\
    &\begin{aligned}\label{0HE3}
      \wit_\ind\rightarrow\bw
      \\
      \bb_\ind\rightarrow\bb
    \end{aligned}
    & \hbox{ strongly in }
    L^p([0,T];  \mathbf{H}_{1}) ,\quad \forall\,  p\in[1,\infty[,
    \\
    &     \begin{aligned}
      \p_t   \wit_\ind\rightarrow\p_t\bw
      \\
      \p_t\bb_\ind\rightarrow\p_t\bb
    \end{aligned}
    &\hbox{ weakly in } L^2([0,T];
    \mathbf{H}_{0}).
  \end{eqnarray}
  
  This already implies that $(\wit,\bb)$
  satisfies~(\ref{eq:Reg11})-(\ref{eq:Reg12}). From~(\ref{0HE3}) and
  the continuity of $D_{N_i}$ in ${\bf H}_s$, we get strong convergence
  of $D_{N_1} (\wit_\ind),\,D_{N_2}(\bb_\ind)$ in $L^4
  ([0,T]\times\tore)$, hence the convergence of the corresponding bi-linear
  products in $L^2 ([0,T] \times \tore)$. This proves that
for all $\vittest, \bh \in L^2([0,T] ; {\bf    H}_1)$
\begin{eqnarray}
  \label{eq:VARFORM}
   &&
   \begin{aligned}
     & \displaystyle\int_0^T\int_{\tore} \p_t \wit\cdot\vittest-
     \int_0^T\int_{\tore} G_1({D_{N_1} (\wit) \otimes D_{N_1}
       (\wit)}):\g\vittest
     \\
     & \qquad + \int_0^T\int_{\tore} G_1({D_{N_2} (\bb) \otimes
       D_{N_2} (\bb)}) : \g \vittest
     \\
     & \qquad+ \nu \int_0^T\int_{\tore} \g \wit : \g \vittest=
     \int_0^T\int_{\tore} G_1({\bf f}) \cdot \vittest,
   \end{aligned}
\\
   &&
   \begin{aligned}
     & \displaystyle\int_0^T\int_{\tore} \p_t \bb \cdot \bh -
     \int_0^T\int_{\tore} G_2({D_{N_2} (\bb) \otimes D_{N_1} (\bw)}) :
     \g  \bh
     \\
     & \qquad + \int_0^T\int_{\tore} G_2({D_{N_1} (\bw) \otimes D_{N_2} (\bb)}) : \g  \bh
     + \mu \int_0^T\int_{\tore} \g \bb : \g  \bh=0.
   \end{aligned}
 \end{eqnarray}
 To introduce the pressure, observe that taking the divergence of the
 equation for $\bw$, we get

  \begin{equation}
    \label{eq:pressure}
    \Delta q = \nabla \cdot G_1\bff+\g \cdot \mathcal{A}_N,
  \end{equation}
  for $ \mathcal{A}_N := -G_1 \big[ \g \cdot \big(D_{N_1} (\wit)
  \otimes D_{N_1} (\wit)\big) - \g \cdot \big(D_{N_2} (\bb) \otimes
  D_{N_2} (\bb)\big) \big]$.  A fairly standard application of De
  Rham's Theorem shows existence of $q$, and the regularity of $
  \mathcal{A}_N $ yields $q \in L^2 ([0,T]; H^1(\tore))$.

  The meaning in which the initial data are taken is completely
  standard and we end the proof by showing uniqueness: Let
  $(\wit_1,\bb_1)$ and $(\wit_2,\bb_2)$ be two solutions corresponding
  to the same data $(\bu_0,\bB_0,\bff)$ and let us define, as usual,
  $\mathcal{W}:=\wit_1-\wit_2$ and $\mathcal{B}:=\bb_1-\bb_2$. By
  standard calculations (mimicking those employed in~\cite{BL2011}),
  we get
 \begin{gather}
    \label{eq:56LKP0}
    \begin{align*}
      \displaystyle \frac{1}{2}\frac{d}{dt} & \Big(
      \|A_1^{\frac{1}{2}} D_{N_1}^{\frac{1}{2}}( \mathcal{W})\|^2 +
      \|A_2^{\frac{1}{2}} D_{N_2}^{\frac{1}{2}}( \mathcal{B}) \|^2
      \Big) +\nu \|\g A_1^{\frac{1}{2}}
      D_{N_1}^{\frac{1}{2}}(\mathcal{W})\|^2 +\mu \| \g
      A_2^{\frac{1}{2}}D_{N_2}^{\frac{1}{2}}( \mathcal{B} )\|^2
      \\
       &=\int_{\tore} (D_{N_2}(\mathcal{B})\cdot\g) D_{N_2} (\bb_1)
      \cdot D_{N_1} (\mathcal{W}) -\int_{\tore}
      (D_{N_1}(\mathcal{W})\cdot\g) D_{N_1}(\bw_1) \cdot
      D_{N_1}(\mathcal{W})
      \\
      &\  +\int_{\tore} (D_{N_2}(\mathcal{B})\cdot\g) D_{N_1}(\bw_1)
      \cdot D_{N_2}(\mathcal{B}) -\int_{\tore}
      (D_{N_1}(\mathcal{W})\cdot\g) D_{N_2} (\bb_1) \cdot
      D_{N_2}(\mathcal{B}),
      \\
       &\leq 2 \| D_{N_2}(\mathcal{B}) \|_{L^4} \| D_{N_1}
      \mathcal{W} \|_{L^4} \| \g D_{N_2} (\bb_1) \|_{L^2} + \| D_{N_1}
      \mathcal{W}\|^2_{L^4} \| \g D_{N_1}(\bw_1) \|_{L^2},
      \\
      &\  + \| D_{N_2}(\mathcal{B}) \|^2_{L^4} \| \g D_{N_1}(\bw_1)
      \|_{L^2}
      \\
       & \leq 2 \| D_{N_2}(\mathcal{B})\|^{1/4}
      \|D_{N_1}(\mathcal{W})\|^{1/4} \| \g
      D_{N_2}(\mathcal{B})\|^{3/4} \|\g D_{N_1}(\mathcal{W}) \|^{3/4}
      \| \g D_{N_2}(\bb_1) \|
      \\
      &\  + \| D_{N_1}(\mathcal{W})\|^{1/2} \|\g
      D_{N_1}(\mathcal{W})\|^{3/2} \| \g D_{N_1}(\bw_1)\| +
      \|D_{N_2}(\mathcal{B})\|^{1/2} \|\g D_{N_2}(\mathcal{B}) \|^{3/2}\| \g D_{N_1}(\bw_1) \|\, .
      \end{align*}
  \end{gather}
 By using $\|D_{N_i}\| = (N_i+1)$, the bound of $\wit_1, \bb_1$ in
  $L^\infty([0,T]; {\bf H}_1)$, and Young's inequality, we obtain
  \begin{equation*}
    \begin{array}{l}
      \displaystyle \frac{1}{2}{d \over  dt}  \big( \| A_1^{\frac{1}{2}} D_{N_1}^{\frac{1}{2}}( \mathcal{W}) \|^2 +
      \| A_2^{\frac{1}{2}} D_{N_2}^{\frac{1}{2}}( \mathcal{B}) \|^2
      \big) +\frac{\nu}{2} \| \g A_1^{\frac{1}{2}} D_{N_1}^{\frac{1}{2}}( \mathcal{W} )\|^2 
      +\frac{\mu}{2} \| \g A_2^{\frac{1}{2}} D_{N_2} ^{\frac{1}{2}}( \mathcal{B} )\|^2
      \\
      \displaystyle \quad \leq{ C (N_1+1)^4 \Big(\sup_{t \ge 0} \| \g
        \wit_1 \|^4 \Big)} \Big[ \frac{1}{\nu^3} \| A_1^{\frac{1}{2}}
      D_{N_1}^{\frac{1}{2}} (\mathcal{W}) 
      \|^2 + \frac{1}{\mu^3} \| A_2^{\frac{1}{2}}
      D_{N_2}^{\frac{1}{2}} (\mathcal{B}) \|^2 \Big] 
      \\
      \displaystyle \qquad +{ C (N_2+1)^4 \Big(\sup_{t \ge 0} \| \g
        \bb_1 \|^4 \Big)} \frac{1}{\nu^{3/2}\mu^{3/2}}\big[  \|
      A_1^{\frac{1}{2}}  D_{N_1}^{\frac{1}{2}} (\mathcal{W}) \|^2 + \|
      A_2^{\frac{1}{2}}  D_{N_2}^{\frac{1}{2}} (\mathcal{B}) \|^2
      \big] . 
    \end{array}
  \end{equation*}
  In particular, we get
  \begin{equation*}
    \frac{1}{2}\frac{d}{dt}  \big( \| A_1^{\frac{1}{2}}D_{N_1}^{\frac{1}{2}}( \mathcal{W}) \|^2+
    \| A_2^{\frac{1}{2}}D_{N_2}^{\frac{1}{2}}( \mathcal{B})\|^2 \big)
    \leq M \big[\|A_1^{\frac{1}{2}}D_{N_1}^{\frac{1}{2}} (\mathcal{W})\|^2+
    \|A_2^{\frac{1}{2}}D_{N_2}^{\frac{1}{2}} (\mathcal{B})\|^2 \big] ,
  \end{equation*}
where
 \begin{gather*}
   M := C \big( \max\big\{ \frac{1}{\nu}, \frac{1}{\mu} \big\} \big)^3
   \Big[ (N_1+1)^4 \Big(\sup_{t \ge 0} \| \g \wit_1 \|^4 \Big) +
   (N_2+1)^4 \Big(\sup_{t \ge 0} \| \g \bb_1 \|^4 \Big) \Big] .
 \end{gather*}
 Since the initial values $\mathcal{W}(0)=\mathcal{B}(0)$ are
 vanishing, we deduce from Gronwall's Lemma that
 $A_1^{\frac{1}{2}}D_{N_1}^{\frac{1}{2}}(\mathcal{W})=
 A_2^{\frac{1}{2}}D_{N_2}^{\frac{1}{2}}(\mathcal{B})={\bf 0}$ and we
 conclude that $\mathcal{W}=\mathcal{B}={\bf 0}$.
\end{proof}
\begin{remark}
  \label{Rem:En-Approx}
The same calculations show also that the following energy equality is
satisfied
  \begin{equation*}
    \begin{aligned}
      & {\frac{1}{2}}\frac{d}{dt} \big( \|
      A_1^{\frac{1}{2}}D_{N_1}^{\frac{1}{2}}( \wit)\|^2 + \|
      A_2^{\frac{1}{2}}D_{N_2}^{\frac{1}{2}}(\bb)\|^2 \big) +\nu\| \g
      A_1^{\frac{1}{2}}D_{N_1}^{\frac{1}{2}}(\wit) \|^2 +\mu\| \g
      A_2^{\frac{1}{2}}D_{N_2}^{\frac{1}{2}}(\bb) \|^2
      \\
      & \qquad\qquad = ( A_2^{\frac{1}{2}}D_{N_1}^{\frac{1}{2}}\big( G_1 {\bf
        f}), A_1^{\frac{1}{2}}D_{N_1}^{\frac{1}{2}}( \wit)
      \big).
    \end{aligned}
  \end{equation*}
  As we shall see in the sequel, it seems that it is not possible to
  pass to the limit $N\to+\infty$ directly in this ``energy equality''
  and some work to obtain an ``energy inequality'' is needed.
\end{remark}
\section{Passing to the limit when $N \to \infty$}
\label{sec:limit}
The aim of this section is the proof of the main result of the paper.
For a given $N\in\N$, we denote by $(\wit_N, \bb_N, q_{N})$ the unique
``regular weak'' solution to Problem~\ref{eq:adm1}, where
$N=\min\{N_1,N_2\}\rightarrow+\infty$.  For the sake of completeness
and to avoid possible confusion between the Galerkin index $\ind$ and
the deconvolution index $N$, we write again the system:
\begin{gather}
  \label{eq:adm-N}
  \begin{split}
    &\partial_t \wit_N +\nabla\cdot G_1 ( {D_{N_1}(\wit_N) \otimes
      D_{N_1}(\wit_N)}) -\nabla\cdot G_1 ( {D_{N_2}(\bb_N) \otimes
      D_{N_2}(\bb_N)})
    \\
    & \hspace{5cm} +\nabla q_{N} - \nu\Delta\wit_N = G_1 \bff \qquad
    \hbox{in} \quad [0,T]\times \tore,
    \\
    \\
    &\partial_t \bb_N +\nabla\cdot G_2 ( {D_{N_2}(\bb_N) \otimes
      D_{N_1}(\wit_N)}) -\nabla\cdot G_2 ( {D_{N_1}(\bw_N) \otimes
      D_{N_2}(\bb_N)})
    \\
    & \hspace{6cm} - \mu\Delta\bb_N =0\qquad \hbox{in} \quad
    [0,T]\times \tore,
    \\
    \\
    &\nabla\cdot\wit_N = \nabla\cdot \bb_N = 0 \qquad \hbox{in} \quad
    [0,T]\times \tore ,
    \\
    &(\wit_N, \bb_N) (0, \x) = (G_1{\vit_0}, G_2 \bB_0) (\x) \qquad
    \hbox{in} \quad \tore. \hspace{1.15cm} \phantom{1}
 \end{split}
\end{gather}
More precisely, for all fixed scales $\alpha_1, \alpha_2>0$, we set
\begin{equation*}
\bw_N=\lim_{\ind\to+\infty}\bw_{\ind,N_1,N_2,\alpha_1, \alpha_2}
\end{equation*}
 and similarly for
$\bb_N$.

\begin{proof}[Proof of Thm.~\ref{thm:Principal}]
  We look for additional estimates, uniform in $N$, to get compactness
  properties about the sequences $\{D_{N_1} (\wit_N),D_{N_2}
  (\bb_N)\}_{N\in\N}$ and $\{\wit_N,\bb_N\}_{N\in\N}$. We then prove
  strong enough convergence results in order to pass to the limit in
  the equation~\eqref{eq:adm-N}, especially in the nonlinear terms.
%
  With the same notation of the previous section, we quote in the
  following table the estimates that we will use for passing to the
  limit. The Table~(\ref{eq:Estimates22}) is organized
  as~\eqref{eq:Estimates} and $\alpha=\min\{\alpha_1,\alpha_2\}$.
  \begin{equation}
    \label{eq:Estimates22}
    \begin{array}
      {  @{}>{\vline \hfill  } c @{\, \, \,  }>{\vline
          \hfill} c @{\, \, \,  }>{\vline \hfill } c @{\, \, \,
        }>{\vline \, \, \,  } c   @{\hfill \vline  }}
      \hline
      \,\,  \hbox{Label} & \hbox{Variable} &\hbox{   bound          } \hskip 2cm ~  \hfill &
      \hfill\hbox{order} \quad ~
      \\
      \hline a & \wit_N, \, \bb_N\, \, \, \, \,  &
      L^\infty([0,T];\mathbf{H}_{0}) \cap L^2([0,T];\mathbf{H}_{1})
      \hskip 0,5 cm ~
      &\, O(1)
      \\
      \hline
      b & \wit_N ,\,\bb_N \, \, \, \, \,  & L^\infty([0,T];
      \mathbf{H}_{1})\cap L^2([0,T];\mathbf{H}_{2})   \hskip 0,5 cm ~&
      \,
      O(\alpha^{-1})
      \\
      \hline \, \,   c & \,\,D_{N_1} (\wit_N), \,D_{N_2} (\bb_N) \, \, &
      L^\infty([0,T];\mathbf{H}_{0})\cap L^2([0,T];\mathbf{H}_{1}) \, &
      \, O(1)
      \\
      \hline \, \,   d &\p_t \wit_N, \,\p_t \bb_N \, \, &\, \,
      L^2([0,T]\times\tore)^3   \hskip 1,2cm ~  & \,
      O(\alpha^{-1} )
      \\
      \hline \, \,   e & q_N \, \, \, \, \, \,  &\,
      L^{2}([0,T];  H^1(\tore)) \cap L^{5/3}([0,T];  W^{2, 5/3}(\tore)) \,
      &\, O(\alpha^{-1} )
      \\
      \hline \, \,   f &\, \, \, \p_t D_{N_1}(\wit_N), \, \p_t
      D_{N_2}(\bb_N)  \, \, &\, \,   L^{4/3}([0,T];{\bf H}_{-1})   \hskip
      1,2cm ~  & \, O(1 )
      \\
      \hline
    \end{array}
  \end{equation}
  Estimates~(\ref{eq:Estimates22}-a),~(\ref{eq:Estimates22}-b),~(\ref{eq:Estimates22}-c),
  and~(\ref{eq:Estimates22}-d) have already been obtained in the
  previous section. Therefore, we just have to
  check~(\ref{eq:Estimates22}-e) and~(\ref{eq:Estimates22}-f).

  \medskip

  {\sl Proof of~(\ref{eq:Estimates22}-e) } --- To obtain further
  regularity properties of the pressure we use again~
  \eqref{eq:pressure}.
  We already know from the estimates proved in the previous section
  that $\mathcal{A}_N \in L^2([0,T] \times \tore)^3$. Moreover,
  classical interpolation inequalities combined
  with~(\ref{eq:Estimates22}-c) yield
  $D_{N_1}(\wit_N),D_{N_2}(\bb_N)\in
  L^{10/3}([0,T]\times\tore)$. Therefore, $\mathcal{A}_N\in
  L^{5/3}([0,T]; W^{1, 5/3}(\tore))$. Consequently, we obtain the
  claimed bound on $q_N$.

  \medskip

  {\sl Proof of~(\ref{eq:Estimates22}-f) } --- Let be given
  $\vittest,\bh\in L^4 ([0,T];{\bf H}_1)$. We use
  $D_{N_1}(\vittest),D_{N_2}(\bh)$ as test functions.  By using that
  $\p_t \wit,\,\p_t\bb \in L^2 ([0,T] \times \tore)^3$, $D_{N_i}$
  commute with differential operators, $G_i$ and $D_{N_i}$ are
  self-adjoint, and classical integrations by parts, we get
  \begin{equation*}
    \begin{aligned}
      (\p_t \wit_N,\, D_{N_1} (\vittest))& = (\p_t D_{N_1}(\wit_N),
      \vittest)
      \\
      &=\nu (\Delta \wit_N, D_{N_1} (\vittest)) + ( D_{N_1} (\wit_N)
      \otimes D_{N_1} (\wit_N), G_1 {D_{N_1} (\g \vittest)})
      \\
      & \quad - ( D_{N_2} (\bb_N) \otimes D_{N_2} (\bb_N), G_1 {D_{N_1} (\g
        \vittest)}) + (D_{N_1} (G_1{ \bf f}), \vittest),
      \\
      \\
      (\p_t \bb_N,\, D_{N_2} (\bh)) &= (\p_t D_{N_2}(\bb_N), \bh)
      \\
      &=\mu (\Delta \bb_N, D_{N_2} (\bh))+(D_{N_2} (\bb_N)\otimes
      D_{N_1}(\wit_N),G_2 {D_{N_2} (\g\bh)})
      \\
      &\quad - ( D_{N_1} (\wit_N) \otimes D_{N_2} (\bb_N), G_2 {D_{N_2}
        (\g\bh)})\, .
    \end{aligned}
  \end{equation*}
  We first observe that
  \begin{equation*}
    \begin{aligned}
      |(\Delta \wit_N, D_{N_1} (\vittest))| = |( \g D_N (\wit_N) ,\g
      \vittest ) | \le C_1 (t) \| \vittest \|_1,
      \\
      |(\Delta \bb_N, D_{N_2} (\bh))| = |( \g D_{N_2} (\bb_N) ,\g \bh
      ) | \le C_1 (t) \| \bh \|_1,
    \end{aligned}
  \end{equation*}
  and that the $L^2([0,T]; H^1(\tore)^3)$ bound for
  $D_{N_1}(\wit_N),\,D_{N_2}(\bb_N)$ implies that $C_1(t) \in
  L^2([0,T])$, uniformly with respect to $N\in\N$.
%
  Therefore, when we combine the latter estimates with the properties
  of $D_{N_i}$ we get, %
 uniformly in $N$,
  \begin{equation*}
    \begin{aligned}
      & |(\p_t D_{N_1}(\wit_N), \vittest)|+ |(\p_t D_{N_2}(\bb_N),
      \vittest)|
      \\
      &\qquad \leq \big( \nu C_1(t)+C_2(t)\big)\|\vittest\|_1+\big(
      \mu C_1(t)+C_2(t)\big) \|\bh\|_1+\|\bff(t,\cdot)\|\in
      L^{4/3}(0,T).
    \end{aligned}
  \end{equation*}
%


%
From the
  estimates~\eqref{eq:Estimates22} and classical rules of functional
  analysis, we can infer that there exist
  \begin{equation*}
    \begin{aligned}
      &\wit, \bb \in L^\infty([0,T]; \mathbf{H}_{1}) \cap
      L^2([0,T];\mathbf{H}_{2}),
      \\
      &{\bf z}_1, {\bf z}_2 \in L^\infty([0,T]; \mathbf{H}_{0}) \cap
      L^2([0,T];\mathbf{H}_{1}),
      \\
      &q\in L^2 ([0,T]; H^1(\tore))\cap
      L^{5/3}([0,T];W^{2,5/3}(\tore))
    \end{aligned}
  \end{equation*}
  such that, up to sub-sequences,
  \begin{equation}
    \label{conv:res}
    \begin{array}{ll}
      \begin{array}{l}
        \wit_N \rightarrow \wit
        \\
        \bb_N \rightarrow \bb
      \end{array}
      &
      \left \{
        \begin{array} {l}
          \hbox{weakly in }  L^2([0,T];  \mathbf{H}_{2}),
          \\
          \hbox{weakly$\ast$ in }  L^\infty([0,T];  \mathbf{H}_{1}),
          \\
          \hbox{strongly in } L^p([0,T]; {\bf H}_1) \qquad \forall \, p\in[1,  \infty[,
        \end{array}
      \right.
      \\
      \\
      \begin{array}{l}
        \p_t \wit_N \rightarrow \p_t \wit
        \\
        \p_t \bb_N \rightarrow \p_t \bb
      \end{array}
      & \hbox{weakly in }
      L^2([0,T] \times \tore),
      \\
      \\
      \begin{array}{l}
        D_{N_1}(\wit_N) \rightarrow {\bf z}_1
        \\
        D_{N_2}(\bb_N) \rightarrow {\bf z}_2
      \end{array}
      &
      \left \{
        \begin{array} {l}
          \hbox{weakly in }  L^2([0,T];  \mathbf{H}_{1}),
          \\
          \hbox{weakly$\ast$ in }  L^\infty([0,T];  \mathbf{H}_{0}),
          \\
          \hbox{strongly in } L^p([0,T] \times \tore)^3 \qquad
          \forall \, p\in[1,  10/3[, 
        \end{array}
      \right.
      \\
      \\
      \begin{array}{l}
        \p_t D_{N_1} (\wit_N) \rightarrow \p_t {\bf z}_1
        \\
        \p_t D_{N_2} (\bb_N) \rightarrow \p_t {\bf z}_2
      \end{array}
      & \hbox{weakly in }
      L^{4/3}([0,T];\mathbf{H}_{-1}),
      \\
      \\
      \begin{array}{l}
        q_{N} \rightarrow q
      \end{array}
      & \hbox{weakly in } L^2([0,T];
      H^1(\tore))\cap L^{5/3}([0,T];  W^{2, 5/3}(\tore)).
    \end{array}
  \end{equation}
  We notice that
  \begin{equation}
    \label{eq:NonLin:8}
    \begin{aligned}
      & D_{N_1}(\wit_N) \otimes D_{N_1}(\wit_N) \longrightarrow {\bf z}_1
      \otimes {\bf z}_1 \quad \hbox{strongly in } L^p([0,T] \times
      \tore)^9 \quad \forall \, p\in[1,  5/3[,
      \\
      & D_{N_2}(\bb_N) \otimes D_{N_2}(\bb_N) \longrightarrow {\bf
        z}_2 \otimes {\bf z}_2 \quad \hbox{strongly in } L^p([0,T]
      \times \tore)^9 \quad \forall \, p\in[1,  5/3[,
      \\
      &D_{N_1}(\wit_N) \otimes D_{N_2}(\bb_N) \longrightarrow {\bf z}_1
      \otimes {\bf z}_2 \quad \hbox{strongly in } L^p([0,T] \times
      \tore)^9 \quad \forall \, p\in[1,  5/3[,
    \end{aligned}
  \end{equation}
  while all other terms in the equation pass easily to the limit as
  well. By using the same identification of the limit used
  in~\cite{BL2011}, we can easily check that ${\bf z}_1=A_1\wit$ and
  ${\bf z}_2=A_2\bb$, ending the proof.
\end{proof}
By using well established results on semicontinuity and adapting
calculations well-known for the NSE, we can prove that the solution
$(\bw,\bb)$ satisfies an ``energy inequality.''
\begin{proposition}
  Let be given $\bu_0, \bB_0 \in \mathbf{H}_{0}$, $\bff\in
  L^2([0,T];\mathbf{H}_{0})$, and let
   $\{(\bw_N, \bb_N, q_{N})\}_{N\in\N}$
 be a (possibly relabelled) sequence of regular weak solutions
  converging to a weak solution $(\bw, \bb, q)$ of the filtered MHD
  equations. Then $(\bw, \bb)$ satisfies the energy inequality~\eqref{eq:energy-inequality}
 in the sense of distributions
(see also~\cite{CF1988,FMRT2001a,Tem2001}). This
     implies that $(\bw, \bb)$ is the average of a weak (in the sense
     of Leray-Hopf) or dissipative solution $(\bu, \bB)$ of the MHD
     equation~\eqref{eq:MHD}. In fact, the energy inequality can also
     be read as
 \begin{equation*}
   \frac{1}{2}{d \over  dt}(\| \bu \|^2 +\|\bB \|^2) + \nu \| \g \bu \|^2
   + \mu \| \g \bB \|^2 \le \langle{\bf f},  \vit\rangle.
 \end{equation*}
\end{proposition}
\begin{proof}
The proof is a straightforward adaption of the one in~\cite{BL2011}. We
start from the energy equality for the approximate model as in
Remark~\ref{Rem:En-Approx} and we observe that
the same arguments as before show also that
 \begin{equation*}
   \left\{
     \begin{aligned}
       D_{N_1}^{1/2}(\bw_N)\to A_1^{1/2}(\bw)
       \\
       D_{N_2}^{1/2}(\bb_N)\to A_2^{1/2}(\bb)
     \end{aligned}
   \right\} \qquad \text{weakly in }L^2([0,T];\mathbf{H}_1).
 \end{equation*}
 Next, due to the assumptions on $\bff$, we have
$   A_1^{-1/2}D^{1/2}_N\bff\to \bff$ {strongly in }$L^2([0,T];\mathbf{H}_0)$
 and, since for all $N\in \N$ we have
 $\bw_N(0)=G_1\bu(0)\in\mathbf{H}_2$  and $\bb_N(0)=G_2\bb(0)\in\mathbf{H}_2$, we get
\begin{equation*}
  \begin{aligned}
  &  \frac{1}{2} ( \| A_1^{1/2} D_{N_1}^{1/2} (\wit_N)(0)\|^2
  + \| A_2^{1/2} D_{N_2}^{1/2} (\bb_N)(0) \|^2) +\int_0^t
  \big(A_1^{-1/2} D_{N_1}^{1/2}( {\bf f}), A_1^{1/2} D_{N_1}^{1/2}(
  \wit_N)\big)\,ds
  \\
  &\overset{N\to+\infty}{\longrightarrow}    \frac{1}{2} (\|A_1 \wit(0)\|^2
  + \| A_2 \bb(0) \|^2) +\int_0^t({\bf f}, A_1\wit)\,ds.
  \end{aligned}
\end{equation*}
Next, we use the elementary inequalities for $\liminf$ and $\limsup$
to infer that
\begin{equation*}
 \begin{aligned}
   & \limsup_{N\to+\infty} \frac{1}{2}\big(\| A_1^{1/2} D_{N_1}^{1/2}
   (\wit_N)(t)\|^2+\| A_2^{1/2} D_{N_2}^{1/2} (\bb_N)(t)\|^2\big)
   \\
   &+ \liminf_{N\to+\infty}\bigg(\nu \int_0^t\| \g A_1^{1/2}
   D_{N_1}^{1/2}(\wit_N) (s)\|^2\,ds+\mu \int_0^t\| \g A^{1/2}_2
   D_{N_2}^{1/2}( \bb_N) (s)\|^2\,ds\bigg)
   \\
   &\leq\frac{1}{2}\big( \| A_1\wit(0) \|^2 + \| A_2\bb(0) \|^2 \big)
   +\int_0^t({\bf f}(s), A_1\wit(s))\,ds.
 \end{aligned}
\end{equation*}
By lower semicontinuity of the norm
and identification of the weak limit, we get the thesis.
\end{proof}
\section{Results for the second model}
\label{sec:existence2}
In this section, we consider the following LES model for MHD, which is
based on filtering only the velocity equation (and on the use of
deconvolution operators):
\begin{equation}
  \label{eq:adm1bis}
  \begin{aligned}
    & \partial_t \bw +\nabla\cdot G_1 \big(D_{N_1}(\bw) \otimes
    D_{N_1}(\bw)\big) - \nabla\cdot G_1\big(\bB \otimes \bB\big) +\nabla
    q-\nu\Delta\bw=G_1\bff,
    \\
    & \partial_t \bB +\nabla\cdot \big(\bB \otimes D_{N_1}(\bw)\big) -
    \nabla\cdot \big(D_{N_1}(\bw) \otimes \bB\big)- \mu\Delta\bB=0,
    \\
    &\nabla\cdot\bw = \nabla\cdot\bB =0,
    \\
    &\wit (0, \x)= G_1{\vit_0} (\x), \qquad \bB(0,\x)=\bB_0(\x),
 \end{aligned}
\end{equation}
and we will work with periodic boundary conditions.  A similar model
in the case without deconvolution has been also studied
in~\cite{CS2010}.  \par
Here we take $\alpha_2=0$, so that $\bb=\bB$ and $A_2=G_2=I$, and
$N_2=0$, so that $D_{N_2}\bB={I}\,\bB=\bB$. We set for simplicity
\begin{equation*}
\alpha=\alpha_1>0,\quad G=G_1,\quad A=A_1, \quad N=N_1.
\end{equation*}
The first aim of this section is to show the changes needed (w.r.t
Thm.~\ref{thm:existence}) to prove the existence of a unique solution
to the system~(\ref{eq:adm1bis}) for a given $N\in\N$, when we assume that
the data are such that
\begin{equation}
  \label{eq:reg-initial-data-bis}
  \vit_0 \in \mathbf{H}_{0}, \quad \bB_0 \in \mathbf{H}_{0}, \quad\text{and}\quad
  {\bf f}\in  L^2([0,T]\times\tore),
\end{equation}
which naturally yields
 $ G_1{\vit_0} \in \mathbf{H}_{2},  \; G_1{\bf
    f}\in L^2([0,T]; \mathbf{H}_{2})$.

We start by defining the notion of what we call a ``regular weak''
solution to this system.
\begin{definition}[``Regular weak'' solution]
  \label{RegSolbis}
  We say that the triple $(\wit, \bB, q)$ is a ``regular weak''
  solution to system~(\ref{eq:adm1bis}) if and only if the three
  following items are satisfied:

  \textcolor{red}{1) \underline{\sc Regularity} }
 \begin{eqnarray}
   && \label{Reg11}
   \wit \in   L^2([0,T];\mathbf{H}_{2})\cap C([0,T];\mathbf{H}_{1}),
   \quad \bB \in L^2([0,T];\mathbf{H}_{1}) \cap
   C([0,T];\mathbf{H}_{0}),
   \\
   && \label{Reg12}
   \p_t \wit \in L^2([0,T];\mathbf{H}_{0}), \quad \p_t \bB \in L^2([0,T];\mathbf{H}_{-1}),
   \\
   && \label{Reg13}
   q\in L^2([0,T]; H^1(\tore)),
 \end{eqnarray}

 \textcolor{red}{2) \underline {\sc Initial data} }
 \begin{equation}
   \label{Initial}
   \displaystyle \lim_{t \rightarrow 0}\|\wit(t, \cdot) -  G_1{\vit_0}  \|_{ \mathbf{H}_{1}} = 0, \qquad
   \lim_{t \rightarrow 0}\|\bB(t, \cdot) -  {\bB_0}  \|_{ \mathbf{H}_{0}} = 0,
 \end{equation}

 \textcolor{red}{3) \underline{\sc Weak Formulation}}: For all $\vittest, \bh  \in L^2([0,T];H^1(\tore)^3)$,
\begin{eqnarray}
   &&
   \begin{aligned}
     & \displaystyle\int_0^T\int_{\tore} \p_t \wit \cdot \vittest-
     \int_0^T\int_{\tore} G_1({D_{N_1}(\wit)\otimes
       D_{N_1}(\wit)}):\g\vittest
     \\
     & \qquad + \int_0^T\int_{\tore} G_1(\bB\otimes \bB):\g\vittest
     +\int_0^T\int_{\tore} \g q\cdot\vittest
     \\
     & \qquad +\nu
     \int_0^T\int_{\tore}\g\wit:\g\vittest=\int_0^T\int_{\tore}
     (G_1{\bf f})\cdot\vittest,
   \end{aligned}
   \\
   &&
   \begin{aligned}
     & \displaystyle\int_0^T\int_{\tore}\p_t\bB\cdot\bh-
     \int_0^T\int_{\tore} (\bB\otimes D_{N_1} (\bw)):\g\bh
     \\
     & \qquad +\int_0^T\int_{\tore} ({D_{N_1}(\bw)\otimes \bB}):\g\bh
     + \mu \int_0^T\int_{\tore}\g\bB:\g\bh=0.
   \end{aligned}
\end{eqnarray}
\end{definition}
\begin{remark}
  Due to the certain symmetry in the equations, it turns out that $\bB$
  has the same regularity of $D_{N} \bw$ (not that of $\bw$).
\end{remark}
All terms in the weak formulation are well-defined. Indeed, the only
term to be checked (which is different from the previous section) is
the bi-linear one involving $\bB\in L^4([0,T];L^3(\tore))^3$ and
$D_{N}(\bw)\in L^\infty([0,T];L^6(\tore))^3$. To this end, we observe
that
\begin{align*}
  & \int_0^T \int_{\tore} (\bB\otimes D_{N}(\bw)):\nabla\bh\leq C
  \int_0^T \|\bB(t)\|_{L^3} \|D_{N}(\bw)(t)\|_{L^6}
  \|\nabla\bh(t)\|_{L^2}
  \\
  & \qquad \leq C_T
  \|\bB\|_{L^4([0,T];L^3)}\|D_{N}(\bw)\|_{L^\infty([0,T]; L^6)}
  \|\nabla\bh\|_{L^2([0,T]; L^2)}.
\end{align*}

\medskip

We have now the following theorem showing that
system~\eqref{eq:adm1bis} is well-posed.
\begin{theorem}
  \label{thm:existence-Bis}
  Assume that~\eqref{eq:reg-initial-data-bis} holds, $\alpha>0$ and
  $N\in\N$ are given. Then, Problem~\eqref{eq:adm1bis} has a unique
  regular weak solution satisfying the energy inequality
  \begin{equation*}
    \frac{d}{dt}
    \big(\|A^{\frac{1}{2}}D_{N}^{\frac{1}{2}}(\wit)\|^2+\|\bB\|^2\big)
    +\nu\| \g A^{\frac{1}{2}}D_{N}^{\frac{1}{2}}(\wit) \|^2
    +\mu\|\g\bB \|^2
   \leq        C ( \| \vit_0 \|, \, \| \bB_0\|, \, \nu^{-1} \| {\bf f}
      \|_{L^2 ([0,T]; \mathbf{H}_{-1}) }).
 \end{equation*}
\end{theorem}
\begin{proof}
  We use the same notation and tools from the previous section and the
  main result can be derived from the energy estimate.  We just give
  some details on the estimates which are different from the previous
  case, since the reader can readily fill the missing details. We use
  $D_{N}(\bw_\ind)$ in the first equation and $\bB_\ind$ in the second
  one as test functions to obtain
\begin{equation*}
  \begin{aligned}
    & \frac{1}{2}\frac{d}{dt}
    \Big(\|A^{1/2}D_{N}^{1/2}(\bw_\ind)\|^2+\|\bB_\ind\|^2\Big)
    +\nu\|\nabla A^{1/2}
    D_{N}^{1/2}(\bw_\ind)\|^2+\mu\|\nabla\bB_\ind\|^2
    \\
    & \qquad = \big(A^{1/2} D_{N}^{1/2}(
    G\mathbf{f}_{1/m}),A^{1/2}D_{N}^{1/2} (\bw_\ind)\big) .
\end{aligned}
  \end{equation*}
  Then, by using the same tools employed in the previous section, we
  have the
  following estimates.
\begin{equation}
  \label{eq:Estimates-bis}
    \begin{array} {  @{}>{\vline \hfill  } c @{\,
          \, \,  }>{\vline  \hfill} c @{\, \, \,  }>{\vline \hfill } c
        @{\, \, \,  }>{\vline \, \, \,  } c   @{\hfill \vline  }}
      \hline
      \,\,  \hbox{Label} & \hbox{Variable} &\hbox{   Bound          }
      \hskip 2cm ~  \hfill &      \hfill\hbox{Order} \quad ~
      \\
      \hline
      \,a)\ \  \phantom{} & \, A^{\frac{1}{2}}D_{N}^{\frac{1}{2}}(\wit_\ind),\,\bB_\ind&
      \,L^\infty([0,T];\mathbf{H}_{0})\cap L^2([0,T];\mathbf{H}_{1})   &
      \hspace{.2cm} O(1)
      \\
      \hline b)\ \  \phantom{}& D_{N}^{1/2}(\wit_\ind) \,  &
      L^\infty([0,T];\mathbf{H}_{0}) \cap L^2([0,T];  \mathbf{H}_{1})   & \hspace{.2cm} O(1)
      \\
      \hline c )\ \  \phantom{}& D_{N}^{1/2}(\wit_\ind)\, \,  &
      L^\infty([0,T];\mathbf{H}_{1})\cap L^2([0,T];  \mathbf{H}_{2})   & \hspace{.2cm}
      O(\alpha^{-1})
      \\
      \hline d)\ \  \phantom{} & \wit_\ind \,\, \, \, \, \,
      &
      L^\infty([0,T];\mathbf{H}_{0})\cap L^2([0,T];\mathbf{H}_{1})   &\hspace{.2cm} O(1)
      \\
      \hline
      e )\ \  \phantom{}& \wit_\ind \, \, \, \, \,  &
      L^\infty([0,T];\mathbf{H}_{1})\cap L^2([0,T];\mathbf{H}_{2})   &
      \hspace{.2cm} O(\alpha^{-1})
      \\
      \hline
      f)\ \  \phantom{} & D_{N} (\wit_\ind) \,
      \, & L^\infty([0,T];\mathbf{H}_{0})\cap
      L^2([0,T];\mathbf{H}_{1})   & \hspace{.2cm} O(1)
      \\
      \hline
      g)\ \  \phantom{} & D_{N} (\wit_\ind) \,
      \, \, & L^\infty([0,T];\mathbf{H}_{1})\cap
      L^2([0,T];\mathbf{H}_{2})   & O(\frac{ \sqrt{N+1}}{\alpha}) \raisebox{3ex} \,~
      \\[.5ex]
      \hline
      h)\ \  \phantom{} &\p_t \wit_\ind \, \, &
      L^2([0,T];\mathbf{H}_{0})   \hspace{.2cm} ~  &  \hspace{.2cm} O(\alpha^{-1}).
      \\
      \hline
      i)\ \  \phantom{} &\p_t \bB_\ind \raisebox{3ex} \, \, &
      L^2([0,T];\mathbf{H}_{-1})   \hspace{.2cm} ~  &  \hspace{.2cm} O(\frac{{(N+1)}^{1/4}}{\alpha^{1/2}}).
      \\[.5ex]
      \hline
    \end{array}
  \end{equation}

  \medskip

  The estimates~(\ref{eq:Estimates-bis}-a)--(\ref{eq:Estimates-bis}-h)
  are the exact analogous of the corresponding ones
  from~\eqref{eq:Estimates}. What it remains to be proved is
  just~(\ref{eq:Estimates-bis}-i). Let be given $\bh \in
  L^2([0,T];\bH_1)$; then
\begin{equation*}
  (\partial_t \bB_\ind, \bh) = -\mu(\nabla\bB_\ind, \nabla \bh) +
  (\bB_\ind\otimes D_{N}(\bw_\ind), \nabla \bh) - (D_{N}(\bw_\ind)\otimes
  \bB_\ind, \nabla \bh)\, .
\end{equation*}
Hence we obtain, by the usual Sobolev and convex interpolation
inequalities,
\begin{equation*}
  \begin{aligned}
    \left| (\partial_t \bB_\ind, \bh) \right|&
    \leq\mu\|\nabla\bB_\ind\|\,\|\nabla
    \bh\|+2\|\bB_\ind\|_{L^6}\|D_N\bw_\ind\|_{L^3}\|\nabla\bh\|
    \\
    & \leq\|\nabla\bB_\ind\|\big(\mu+
    C\|D_N\bw_\ind\|^{1/2} \|\g
    D_N\bw_\ind\|^{1/2}\big)\|\nabla\bh\|.
  \end{aligned}
\end{equation*}
Next, by employing estimates~(\ref{eq:Estimates-bis}-a)-d)-f)-g), we
get
\begin{equation*}
  \begin{aligned}
    &\left|\int_0^T (\partial_t \bB_\ind, \bh)\,dt \right|
    \\
    &\leq\|\bB_\ind\|_{L^2([0,T];\bH_{1})}\big(\mu+
    \|D_{N}(\bw_\ind)\|_{L^\infty([0,T];\bH_0)}^{1/2}\|\nabla
    D_{N}(\bw_\ind)\|_{L^\infty([0,T];\bH_0)}^{1/2}
    \big)\|\nabla\bh\|_{L^2([0,T];L^2)},
    \\
    & \leq C\big(\mu+\frac{(N+1)^{1/4}}{\alpha^{1/2}}
    \big)\|\nabla\bh\|_{L^2([0,T];L^2)}.
  \end{aligned}
\end{equation*}
These estimates are enough to pass to the limit as $\ind\to+\infty$
and to show that the limit $(\bw,\bB)$ is a weak solution which satisfies
\begin{eqnarray}
  \label{eq:VARFORMbis}
   &&
   \begin{aligned}
     & \displaystyle\int_0^T\int_{\tore} \p_t \wit\cdot\vittest-
     \int_0^T\int_{\tore} G({D_{N} (\wit) \otimes D_{N}
       (\wit)}):\g\vittest + \int_0^T\int_{\tore} G({\bB\otimes\bB}):\g\vittest
     \\
     & \qquad+ \nu \int_0^T\int_{\tore} \g \wit : \g \vittest=
     \int_0^T\int_{\tore} G({\bf f}) \cdot \vittest,
   \end{aligned}
\\
   &&
   \begin{aligned}
     & \displaystyle\int_0^T\int_{\tore} \p_t \bB \cdot \bh -
     \int_0^T\int_{\tore} ({ \bB \otimes D_{N} (\bw)}):\g\bh
     \\
     & \qquad + \int_0^T\int_{\tore} ({D_{N} (\bw) \otimes\bB}):\g\bh
     + \mu \int_0^T\int_{\tore} \g \bB : \g  \bh=0.
   \end{aligned}
 \end{eqnarray}
 The introduction of the pressure follows exactly as in the previous
 section, while the uniqueness needs some minor adjustments.  Let in
 fact $(\wit_1,\bB_1)$ and $(\wit_2,\bB_2)$ be two solutions
 corresponding to the same data $(\bu_0, \bB_0, \bff)$ and let us
 define as usual $\mathcal{W}:=\wit_1-\wit_2$ and $\mathcal{B}:=\bB_1-
 \bB_2$. We will use $AD_{N} (\mathcal{W})$ and $\mathcal{B}$ as test
 functions in the equations satisfied by $\mathcal{W}$ and $\mathcal{B}$,
 respectively. Observe that, by standard calculations, $AD_{N}
 (\mathcal{W})$ lives in $L^2 ([0,T]\times\tore)^3$, while $\mathcal{B}\in
 L^2([0,T];\bH_{1})$.  In order to justify the calculations ---~those
 for the velocity equation are analogous to the previous ones~--- first
 observe that, for any fixed order of deconvolution $N$,
\begin{equation*}
  \int_0^t\langle\p_t\mathcal{B},\mathcal{B}\rangle_{\bH_1,\bH_{-1}}=
  \frac{1}{2}\big(\|\mathcal{B}(t)\|^2-\|\mathcal{B}(0)\|^2\big),
\end{equation*}
since the duality is well-defined thanks to~(\ref{eq:Estimates-bis}-a)-i).
We formally write the distributional expression, keeping the time
derivative, and  we get the following equality (to be more precise, one
should write directly the integral formula, after integration over
$[0,t]$, but the reader can easily fill the details):
\begin{equation*}
  \begin{aligned}
    & \frac{1}{2} \frac{d}{dt} \big(\|A^{1/2}D_{N}^{1/2}(\mathcal{W})\|^2
    +\|\mathcal{B}\|^2 \big) + \nu\|\nabla
    A^{1/2}D_{N}^{1/2}(\mathcal{W})\|^2+
    \mu\|\nabla\mathcal{B}\|^2
    \\
    & \qquad = -\big( (D_{N}(\mathcal{W}) \cdot\nabla)D_{N}(\bw_1),\,
    D_{N}(\mathcal{W})\big) + \big((\mathcal{B}\cdot\nabla)\bB_1,\, D_{N}(\mathcal{W})\big)
    \\
    & \qquad \qquad -\big((D_{N}(\mathcal{W})\cdot\nabla)\bB_1,\,\mathcal{B}\big)
    +\big((\mathcal{B}\cdot \nabla) D_{N}(\bw_1),\, \mathcal{B}\big)
    \\
    & \qquad =: I_1 + I_2 + I_3 + I_4\, .
  \end{aligned}
\end{equation*}
Now, we need to estimate the four integrals in the right-hand
side. The estimates are obtained by using the standard interpolation
and Sobolev inequalities together with the properties of $D_N$. %
%
%
We have:
\begin{equation*}
\begin{aligned}
  |I_1| \leq \varepsilon\nu\|\nabla A^{1/2}
  D_{N}^{1/2}(\mathcal{W})\|^2+\frac{C_\varepsilon
    (N+1)^4 \sup_{t\geq0}\|\nabla\bw_1\|^4}{\nu^3} \|A^{1/2}
  D_{N}^{1/2}(\mathcal{W})\|^2\, ,
\end{aligned}
\end{equation*}

\medskip

\noindent

\begin{equation*}
  \begin{aligned}
    |I_2| & \leq \|\mathcal{B}\|_{L^4} \|\nabla D_{N}(\mathcal{W})\|_{L^4} \|\bB_1\|
    \\
    & \leq C\|\mathcal{B}\|^{1/4} \|\nabla\mathcal{B}\|^{3/4} \|\nabla
    D_{N}(\mathcal{W})\|^{1/4} \|\Delta D_{N}(\mathcal{W})\|^{3/4}\|\bB_1\|
    \\
    & \leq C \frac{(N+1)^{1/2}}{\alpha} \|\mathcal{B}\|^{1/4}
    \|\nabla\mathcal{B}\|^{3/4} \alpha^{1/4}\|\nabla D_{N}^{1/2}(\mathcal{W})\|^{1/4}
    \alpha^{3/4}\|\Delta D_{N}^{1/2}(\mathcal{W})\|^{3/4}\|\bB_1\|
    \\
    & \leq C \frac{(N+1)^{1/2}}{\alpha} \|\mathcal{B}\|^{1/4}
    \|\nabla\mathcal{B}\|^{3/4} \|A^{1/2} D_{N}^{1/2}(\mathcal{W})\|^{1/4} \|\nabla
    A^{1/2} D_{N}^{1/2}(\mathcal{W})\|^{3/4}\|\bB_1\|
    \\
    & \leq \varepsilon\mu\|\nabla\mathcal{B}\|^2 + \frac{C_\varepsilon
      (N+1)^{4/5}}{\mu^{3/5}\alpha^{8/5}} \|\mathcal{B}\|^{2/5} \|A^{1/2}
    D_{N}^{1/2}(\mathcal{W})\|^{2/5} \|\nabla A^{1/2}
    D_{N}^{1/2}(\mathcal{W})\|^{6/5}\|\bB_1\|^{8/5}
    \\
    & \leq \varepsilon\mu\|\nabla\mathcal{B}\|^2 + \varepsilon\nu\|\nabla
    A^{1/2} D_{N}^{1/2}(\mathcal{W})\|^2 + \frac{C_\varepsilon
      (N+1)^2}{\mu^{3/2}\nu^{3/2}\alpha^4} \|\mathcal{B}\|\, \|A^{1/2}
    D_{N}^{1/2}(\mathcal{W})\| \, \|\bB_1\|^4
    \\
    & \leq \varepsilon\mu\|\nabla\mathcal{B}\|^2 + \varepsilon\nu\|\nabla
    A^{1/2} D_{N}^{1/2}(\mathcal{W})\|^2 + \frac{C_\varepsilon
      (N+1)^2}{\mu^{3/2}\nu^{3/2}\alpha^4} \|\bB_1\|^4 \big( \|\mathcal{B}\|^2
    + \|A^{1/2} D_{N}^{1/2}(\mathcal{W})\|^2 \big) .
\end{aligned}
\end{equation*}

\bigskip

\begin{equation*}
\begin{aligned}
  |I_3| & 
  \leq \|D_{N}(\mathcal{W})\|_{L^\infty} \|\nabla\mathcal{B}\| \, \|\bB_1\|
  \\
  & \leq C \|\nabla D_{N}(\mathcal{W})\|^{1/2} \|\Delta D_{N}(\mathcal{W}) \|^{1/2}
  \|\nabla\mathcal{B}\| \, \|\bB_1\|
  \\
  & \leq C \frac{(N+1)^{1/2}}{\alpha} \|A^{1/2}
  D_{N}^{1/2}(\mathcal{W})\|^{1/2} \|\nabla A^{1/2} D_{N}^{1/2}(\mathcal{W})\|^{1/2}
  \|\nabla\mathcal{B}\| \, \|\bB_1\|
  \\
  & \leq \varepsilon\mu\|\nabla\mathcal{B}\|^2 +
  \frac{C_\varepsilon(N+1)}{\mu\alpha^2} \|A^{1/2} D_{N}^{1/2}(\mathcal{W})\|
  \, \|\nabla A^{1/2} D_{N}^{1/2}(\mathcal{W})\| \, \|\bB_1\|^2
  \\
  & \leq \varepsilon\mu\|\nabla\mathcal{B}\|^2 + \varepsilon\nu \|\nabla
  A^{1/2} D_{N}^{1/2}(\mathcal{W})\|^2 +
  \frac{C_\varepsilon(N+1)^2}{\mu^2\nu\alpha^4} \|\bB_1\|^4 \|A^{1/2}
  D_{N}^{1/2}(\mathcal{W})\|^2\, ,
\end{aligned}
\end{equation*}

\bigskip

\begin{equation*}
\begin{aligned}
  |I_4| & \leq \|\mathcal{B}\|_{L^4}^2 \|\nabla D_{N}(\mathcal{W})\| \leq C
  \|\mathcal{B}\|^{1/2}\|\nabla\mathcal{B}\|^{3/2}(N+1)\|\nabla\bw_1\|
  \\
  & \leq \varepsilon\mu\|\nabla\mathcal{B}\|^2+\frac{C_\varepsilon
    (N+1)^4}{\mu^3} \|\nabla\bw_1\|^4 \|\mathcal{B}\|^2\, .
\end{aligned}
\end{equation*}
We then set $\varepsilon=1/6$ and, by collecting all the estimates, we
finally obtain
\begin{equation*}
  \frac{d}{dt} \big( \|A^{1/2} D_{N}^{1/2}(\mathcal{W})\|^2+\|\mathcal{B}\|^2 \big)
  + \nu\|\nabla A^{1/2} D_{N}^{1/2}(\mathcal{W})\|^2 +
  \mu\|\nabla\mathcal{B}\|^2 \leq CM \big( \|A^{1/2} D_{N}^{1/2}(\mathcal{W})\|^2+\|\mathcal{B}\|^2
  \big) ,
\end{equation*}
where
\begin{equation*}
  M \doteq (N+1)^4 \max\sup_{t\geq0}\Big\{
  \frac{\|\nabla\bw_1(t)\|^4}{\nu^3},
  \frac{\|\nabla\bw_1(t)\|^4}{\mu^3},
  \frac{\|\bB_1(t)\|^4}{\mu^{3/2}\nu^{3/2}\alpha^4},\frac{\|\bB_1(t)\|^4}{\mu^2\nu
    \alpha^4}
  \Big\} .
\end{equation*}
An application of the Gronwall's lemma proves (for any fixed $N$) that
$\|A^{1/2} D_{N}^{1/2}(\mathcal{W})\|^2+\|\mathcal{B}\|^2=0$; hence, by using the
properties of $A$ and $D_N$ exploited before, we
finally get
 $ \mathcal{W}=\mathcal{B}\equiv{\bf 0}$.
\end{proof}

\medskip

We can now pass to the problem of the convergence as $N\to+\infty$,
proving the counterpart of Theorem~\ref{thm:Principal}.
\begin{theorem}
  \label{thm:Principal-Bis}
  From the sequence $\{(\wit_N, \bB_N,q_N)\}_{N \in \N}$, one can
  extract a sub-sequence (still denoted $\{(\wit_N,\bB_N,q_N)\}_{N \in
    \N}$) such that
  \begin{alignat*}{2}
    & \wit_N \longrightarrow \wit
    &\quad &
    \left \{
      \begin{array} {l}
        \hbox{weakly in }  L^2([0,T];  \mathbf{H}_{2}),
        \\
        \hbox{weakly$\ast$ in }  L^\infty([0,T];  \mathbf{H}_{1}),
        \\
        \hbox{strongly in } L^p([0,T]; {\bf H}_1) \quad \forall \, p\in[1,  \infty[,
      \end{array}
    \right.
    \\
    &    \bB_N \longrightarrow \bB
    &\quad&
    \left \{
      \begin{array} {l}
        \hbox{weakly in }  L^2([0,T];  \mathbf{H}_{1}),
        \\
        \hbox{weakly$\ast$ in }  L^\infty([0,T];  \mathbf{H}_{0}),
        \\
        \hbox{strongly in } L^p([0,T]\times\tore)^3 \quad \forall \, p\in[1,  10/3[,
      \end{array}
    \right. \\
    &q_N\longrightarrow q&\quad &\text{weakly in }L^2([0,T];H^1(\tore))\cap
    L^{5/3}([0,T]; W^{2,5/3}(\tore)),
  \end{alignat*}
  and such that the system
  \begin{equation}
    \label{eq:adm2-bis}
    \begin{aligned}
      \partial_t\bw+\nabla\cdot G ({A\bw\otimes
        A\bw})-\nabla\cdot G{(\bB\otimes
        \bB)}+\nabla  q -\nu\Delta\bw=\bff,
      \\
      \nabla\cdot\bw=   \nabla\cdot\bB=0,
      \\
      \partial_t\bB+\nabla\cdot {(\bB\otimes A\bw)}-\nabla\cdot
      {(A\bw\otimes\bB)}-\mu\Delta\bb=0,
      \\
      \bw (0, \x) = G{\vit_0} (\x), \qquad \bB (0, \x) = {\bB}_0 (\x)
    \end{aligned}
  \end{equation}
  holds in distributional sense and the following energy inequality is
  satisfied:
  \begin{equation*}
   \frac{1}{2}\frac{d}{dt}(\|A\wit\|^2+\|\bb\|^2)+\nu\|\g A\wit\|^2
   +\mu\|\g\bb\|^2\le({\bf f},A\wit).
 \end{equation*}
\end{theorem}
\begin{proof}
  This result is based on the following estimates and from compactness results.
  \begin{equation}
  \label{eq:Estimates-2bis}
    \begin{array} {  @{}>{\vline \hfill  } c @{\,
          \, \,  }>{\vline  \hfill} c @{\, \, \,  }>{\vline \hfill } c
        @{\, \, \,  }>{\vline \, \, \,  } c   @{\hfill \vline  }}
      \hline
       \hbox{Label} & \hbox{Variable} &\hbox{   Bound          }
      \hskip 2cm ~  \hfill &      \hfill\hbox{Order} \quad ~
      \\
      \hline
      \hline a)& \wit_N 
      &
      L^\infty([0,T];\mathbf{H}_{0})\cap L^2([0,T];\mathbf{H}_{1})   &\hspace{.1cm} O(1)
      \\
      \hline
      b)& \wit_N 
      & L^\infty([0,T];\mathbf{H}_{1})\cap L^2([0,T];\mathbf{H}_{2})   &
      \hspace{.1cm} O(\alpha^{-1})
      \\
      \hline
      c) & D_{N} (\wit_N) ,
      \, \bB_N & L^\infty([0,T];\mathbf{H}_{0})\cap
      L^2([0,T];\mathbf{H}_{1})   & \hspace{.1cm} O(1)
      \\
      \hline
      d) & q_N &\,
   L^{2}([0,T];  H^1(\tore)) \cap L^{5/3}([0,T];  W^{2, 5/3}(\tore)) \,
   &\, O(\alpha^{-1} )
   \\
   \hline
      e) &\p_t \wit_N  &
      L^2([0,T];\mathbf{H}_{0})    &  \hspace{.1cm} O(\alpha^{-1})
      \\
      \hline
      f)&\,\p_t D_N(\bw_N), \p_t \bB_N  &
      L^{4/3}([0,T];\mathbf{H}_{-1})    &  \hspace{.1cm} O(1)
      \\
      \hline
    \end{array}
  \end{equation}
  The only new bound here is represented by the one for $\p_t \bB_N $
  from~(\ref{eq:Estimates-2bis}-f). In fact, by the usual
  interpolation inequalities, we get
\begin{equation*}
  \begin{aligned}
    \left| (\partial_t \bB_N, \bh) \right|&
    \leq\mu\|\nabla\bB_N\|\,\|\nabla
    \bh\|+2\|\bB_N\|_{L^4}\|D_N(\bw_N)\|_{L^4}\|\nabla\bh\|
    \\
    & \leq\big(\mu\|\nabla\bB_N\|+
    C\|\bB_N\|^{1/4}\|\g\bB_N\|^{3/4}\|D_N(\bw_N)\|^{1/4} \|\g
    D_N(\bw_N)\|^{3/4}\big)\|\nabla\bh\|.
  \end{aligned}
\end{equation*}
Next, by employing estimate~(\ref{eq:Estimates-2bis}-c), we get
\begin{equation*}
  \left| (\partial_t \bB_N, \bh) \right| \leq\big(\mu\|\nabla\bB_N\|+
  C\|\g\bB_N\|^{3/4}\|\g    D_N(\bw_N)\|^{3/4}\big)\|\nabla\bh\|,
\end{equation*}
and since both $\g\bB_N,\, \g D_N(\bw)\in L^2([0,T];L^2(\tore)^9)$, we can show that
\begin{equation*}
  \begin{aligned}
    \left|\int_0^T (\partial_t \bB_N, \bh)\,dt \right|
    & \leq
    \mu\|\nabla\bB_N\|_{L^2([0,T];L^2)}\|\nabla\bh\|_{L^2([0,T];L^2)}
    \\
    & \qquad + C \|\nabla\bB_N\|_{L^2([0,T];L^2)}^{3/4} \|\nabla
    D_{N}(\bw)_N \|_{L^2([0,T];L^2)}^{3/4}
    \|\nabla\bh\|_{L^4([0,T];L^2)}\, ,
  \end{aligned}
\end{equation*}
thus proving that $\p_t\bB_N\in L^{4/3}([0,T];\bH_{-1})$,
independently of $N$.

The limit $N\to+\infty$ can be studied as in the previous section. In
addition to the same estimates proved before, from the bound on the
time derivative of $\bB$ we obtain that %
 $ \bB_N\rightarrow\bB$  in $L^p([0,T];\bH_0)$, $\forall\,p\in[1,\infty[$,
and reasoning as in~\eqref{eq:NonLin:8} we get
\begin{equation*}
  \begin{aligned}
     D_{N}(\wit_N) \otimes D_{N}(\wit_N) &\longrightarrow A\bw \otimes
    A\bw \quad \hbox{strongly in } L^p([0,T] \times \tore)^9 \quad
    \forall \, p\in[1,  5/3[,
    \\
    \bB_N\otimes\bB_N &\longrightarrow \bB\otimes\bB\quad
    \hbox{strongly in } L^p([0,T] \times \tore)^9 \quad \forall \, p\in[1,5/3[,
    \\
    D_{N}(\wit_N) \otimes \bB_N &\longrightarrow A\bw\otimes\bB\quad
    \hbox{strongly in } L^p([0,T] \times \tore)^9 \quad \forall \, p\in[1, 5/3[.
  \end{aligned}
\end{equation*}
Finally, the proof of the energy inequality follows  the
same steps as before.
\end{proof}

%
\def\ocirc#1{\ifmmode\setbox0=\hbox{$#1$}\dimen0=\ht0 \advance\dimen0
  by1pt\rlap{\hbox to\wd0{\hss\raise\dimen0
  \hbox{\hskip.2em$\scriptscriptstyle\circ$}\hss}}#1\else {\accent"17 #1}\fi}
  \def\polhk#1{\setbox0=\hbox{#1}{\ooalign{\hidewidth
  \lower1.5ex\hbox{`}\hidewidth\crcr\unhbox0}}} \def\cprime{$'$}

\end{document}